\documentclass[12pt,reqno]{amsart}

\usepackage[usenames, dvipsnames, svgnames]{xcolor}
\usepackage{amsmath, amssymb,graphicx,amsthm,latexsym, amsfonts, enumitem, mathtools, tensor, MnSymbol}
\usepackage{hyperref}
\usepackage[all, color]{xy}
\usepackage{color}
\usepackage{amssymb, wasysym}
\usepackage{float}
\usepackage{tikz}
\usetikzlibrary{arrows,decorations.pathmorphing,decorations.pathreplacing,positioning,shapes.geometric,shapes.misc,decorations.markings,decorations.fractals,calc,patterns}

\DeclareMathOperator{\Hom}{Hom}
\DeclareMathOperator{\End}{End}

\DeclareMathOperator{\Mod}{Mod}
\DeclareMathOperator{\mmod}{mod}

\DeclareMathOperator{\rad}{rad}
\DeclareMathOperator{\Ext}{Ext}

\DeclareMathOperator{\perf}{perf}
\DeclareMathOperator{\Tor}{Tor}

\theoremstyle{plain}
\newtheorem{theorem}{Theorem}[section]
\newtheorem*{theorem*}{Theorem}

\theoremstyle{definition}
\newtheorem{defn}[theorem]{Definition}

\newtheorem{remark}[theorem]{Remark}

\newtheorem{lemma}[theorem]{Lemma}

\newtheorem{setup}[theorem]{Setup}
\newtheorem{proposition}[theorem]{Proposition}

\date{}
\begin{document}
\setlength{\parindent}{0pt}
\setlength{\parskip}{7pt}
%\address{VERONA}
\author{Francesca Fedele}
\address{Dipartimento di Informatica - Settore di Matematica Universita` degli Studi di Verona
Strada le Grazie 15 - Ca` Vignal, I-37134 Verona, Italy}
\email{francesca.fedele@univr.it}

\title[Universal localizations of $d$-homological pairs]{Universal localizations of $d$-homological pairs}

\begin{abstract}
    Let $k$ be an algebraically closed field and $\Phi$ a finite dimensional $k$-algebra. The universal localization $\Phi\rightarrow \Phi_\mathcal{S}$ of  $\Phi$ with respect to a set of morphisms between finitely generated projective $\Phi$-modules $\mathcal{S}$ always exists. Moreover, when $\Phi$ is hereditary, Krause and \v{S}\v{t}ov\'i\v{c}ek proved that the universal localizations of $\Phi$ are in bijective correspondence with various natural structures. Taking inspiration from an alternative definition of universal localizations involving a triangulated subcategory of $\mathcal{D}^{\perf}(\Phi)$, we introduce a higher analogue of universal localizations. That is, fixing a positive integer $d$, we define universal localizations of $d$-homological pairs $(\Phi,\mathcal{F})$ with respect to suitable wide subcategories $\mathcal{U}$ of $\mathcal{D}^b(\mmod\Phi)$. When gldim$\Phi\leq d$, we show that the result by Krause and \v{S}\v{t}ov\'i\v{c}ek has a (partial) higher analogue and that such universal localizations exist with respect to any choice of $\mathcal{U}$ with the required properties. Moreover, we show that in this setup, the base case of our definition and the definition of classic universal localization coincide. 
    \end{abstract}

\maketitle

\section{Introduction}
Let $(\Phi, \mathcal{F})$ be a \textit{$d$-homological pair}, that is $\Phi$ is a finite dimensional $k$-algebra and $\mathcal{F}\subseteq\mmod\Phi$ is a $d$-cluster tilting subcategory for some fixed integer $d\geq 1$. In particular, $\mathcal{F}$ is a $d$-abelian category as defined by Jasso in \cite{GJ}, and plays the role of a ``higher'' $\mmod\Phi$ and if $d=1$, then $\mathcal{F}=\mmod\Phi$ is the only choice.
In this paper, we introduce the definition of universal localizations of  $d$-homological pairs $(\Phi, \mathcal{F})$ and prove that the set of such localizations (up to equivalence) is in bijective correspondence with various natural structures.
We work in the following setup throughout the paper.

\begin{setup}\label{setup}
Let $d$ be a positive integer, $k$ an algebraically closed field and $\Phi$ a finite dimensional $k$-algebra with global dimension at most $d$. The category of finitely generated right $\Phi$-modules is denoted by $\mmod\Phi$ and its bounded derived category by $\mathcal{D}^b(\mmod\Phi)$.
Assume moreover that $\mathcal{F}\subseteq \mmod\Phi$ is a $d$-cluster tilting subcategory in the sense of Definition \ref{defn_dct}, that is $(\Phi, \mathcal{F})$ is a $d$-homological pair.
\end{setup}

\subsection{Classic background ($d=1$ case)}
In the case $d=1$, we have that $\Phi$ is hereditary.

We recall the definition of universal localization of $\Phi$, see \cite[Section 4]{S} or \cite[Section 6]{KS}. Note that this can actually be defined for an associative ring with unit instead of an algebra.

\begin{defn}[{=Definition \ref{defn_classic}}]\label{defn_univ_classic}
A \textit{universal localization of $\Phi$} with respect to a set $\mathcal{S}$ of morphisms between finitely generated projective right $\Phi$-modules is an algebra homomorphism $\Phi\rightarrow \Phi_\mathcal{S}$ initial with respect to the property that every element $\sigma\otimes_\Phi \Phi_\mathcal{S}$ with $\sigma\in\mathcal{S}$ has an inverse.
\end{defn}
Schofield proved in \cite[Theorem 4.1]{S} that a universal localization of $\Phi$ with respect to any such set $\mathcal{S}$ always exists and it is unique up to unique isomorphism. Hence we call it \textit{the} universal localization of $\Phi$ with respect to $\mathcal{S}$.

In \cite[Theorem B]{KS}, Krause and \v{S}\v{t}ov\'i\v{c}ek proved that when $\Phi$ is hereditary, universal localizations of $\Phi$ are in bijective correspondence with various natural structures. We recall their result, which we will then partially generalise to higher cases.

\begin{theorem}[{\cite[Theorem B]{KS}}]\label{thmKS}
For $\Phi$ hereditary, there are bijections between the following sets:
\begin{enumerate}
\item Wide subcategories of $\mmod\Phi$.
\item Wide subcategories of $\mathcal{D}^b(\mmod\Phi)$.
\item Equivalence classes of homological epimorphisms $\Phi\rightarrow\Gamma$.
\item Equivalence classes of universal localizations $\Phi\rightarrow\Gamma$.
\item Wide subcategories of $\Mod\Phi$ that are closed under products and coproducts.
\item Localizing subcategories of $\mathcal{D}(\Mod\Phi)$ that are closed under products.
\item Equivalence classes of localization functors $\mathcal{D}(\Mod\Phi)\rightarrow \mathcal{D}(\Mod\Phi)$ preserving coproducts.
\end{enumerate}
\end{theorem}

In the above, a wide subcategory of an abelian category is a full subcategory closed under direct summands, kernels, cokernels and extensions. A wide subcategory of a triangulated category is a full subcategory closed under $\Sigma^{\pm 1}$, direct summands and extensions. Localizing subcategories and localization functors are beyond the interest of this paper, see \cite{KS} for details.

%In \cite{NR}, Neeman and Ranicki proved a result that links universal localizations to Verdier quotients. Note that again, the construction still makes sense if we have an associative ring with unit instead of the $k$-algebra $\Phi$.

Let $\mathcal{S}$ be a set of maps between finitely generated projective right $\Phi$-modules. Let $\mathcal{D}^{\perf}(\Phi)$ be the homotopy category whose objects are bounded complexes of finitely generated projective $\Phi$-modules, that is \textit{perfect complexes}, and whose morphisms are homotopy equivalence classes of chain maps.
The set $\mathcal{S}$ can be viewed as a set of objects in $\mathcal{D}^{\perf}(\Phi)$, where each element of $\mathcal{S}$ is a complex concentrated in two degrees. Let $\mathcal{R}^c_{\mathcal{S}}$ be the smallest triangulated subcategory of $\mathcal{D}^{\perf}(\Phi)$ containing $\mathcal{S}$ and all the direct summands of any of its objects.% and let $\mathcal{T}^c$ be the idempotent completion of the Verdier quotient $\mathcal{D}^{\perf}(\Phi)/\mathcal{R}^c$.

\begin{remark}\label{remark_B}
The above construction can be used to give an alternative definition of the universal localization. That is, the universal localization of $\Phi$ with respect to the set $\mathcal{S}$ is the algebra homomorphism $\Phi\rightarrow \Phi_\mathcal{S}$ initial in the category of algebra homomorphisms $\Phi\rightarrow\Lambda$ with the property that $\mathcal{R}^c_{\mathcal{S}}\otimes^L_\Phi\Lambda=0$, see Remark \ref{remark_alternative_defn} and Definition \ref{defn_alternative}.
\end{remark}

%\begin{theorem}[{\cite[Theorem 0.7]{NR}}]\label{thmC}
% Consider the natural functor $\mathcal{D}^{\perf}(\Phi)\rightarrow \mathcal{D}^{\perf}(\Phi_\mathcal{S})$ which takes a complex $C$ in $\mathcal{D}^{\perf}(\Phi)$ and sends it to $C\otimes_\Phi\Phi_\mathcal{S}$. Take the canonical factorization
% \begin{align*}
%     \mathcal{D}^{\perf}(\Phi)\xrightarrow{\pi} \mathcal{T}^c\xrightarrow{T}\mathcal{D}^{\perf}(\Phi_\mathcal{S}).
% \end{align*}
%Then the following are equivalent:
%\begin{enumerate}
%    \item The functor $T$ is an equivalence of categories.
%    \item For all $n>0$, $\Tor^\Phi_n(\Phi_\mathcal{S},\Phi_\mathcal{S})=0$.
%\end{enumerate}
%\end{theorem}

%\begin{remark}\label{remarkD}
%Note that in the case when $\Phi$ is hereditary, Theorem \ref{thmKS} implies that any universal localization $\Phi\rightarrow\Phi_\mathcal{S}$ is also a homological epimorphism (and viceversa). In particular, $\Tor^\Phi_n(\Phi_\mathcal{S},\Phi_\mathcal{S})=0$ for all $n>0$. Hence, by Theorem \ref{thmC}, there is an equivalence of categories
%\begin{align*}
%   T:\mathcal{T}^c\xrightarrow{\sim} \mathcal{D}^{\perf} (\Phi_\mathcal{S}).
%\end{align*}
%\end{remark}

\subsection{This paper ($d\geq 2$ case)}

The $d$-cluster tilting subcategory $\mathcal{F}\subseteq \mmod\Phi$ is a $d$-abelian category by \cite[Theorem 3.16]{GJ} in the sense of Definition \ref{defn_dabelian}. Moreover, by \cite[Theorem 1.21]{I},
\begin{align*}
\overline{\mathcal{F}}:=\text{add}\{ \Sigma^{di}\mathcal{F}\mid i\in\mathbb{Z} \}\subseteq\mathcal{D}^b(\mmod\Phi)
\end{align*}
is a $d$-cluster tilting subcategory. Since $\overline{\mathcal{F}}$ is also closed under $\Sigma^d$, by \cite[Theorem 1]{GKO}, we have that $\overline{\mathcal{F}}$ is a $(d+2)$-angulated category containing a copy of $\mathcal{F}$ in the sense of Definition \ref{defn_angulated}.
Note that $\mathcal{F}$ plays the role of a higher $\mmod \Phi$ and $\overline{\mathcal{F}}$ the one of a higher derived category of $\mathcal{F}$. Moreover, the base case $d=1$ recovers the classic theory as $\mathcal{F}=\mmod\Phi$ is the only option and then $\overline{\mathcal{F}}=\mathcal{D}^b(\mmod\Phi)$.

We generalise the definition of universal localization to the one of universal localization of the $d$-homological pair $(\Phi, \mathcal{F})$. Note that, since we work over categories of finitely generated modules and our algebras always have finite global dimension, we could equivalently write $\mathcal{D}^{\perf}$ instead of $\mathcal{D}^b$ in our arguments. This choice would be more consistent with respect to the $d=1$ case, but we choose to write $\mathcal{D}^b$ instead for consistency with the usual setup of higher homological algebra.

\textbf{Definition (= Definition \ref{defn_higher_univ_loc}).} Let $\mathcal{U}\subseteq\mathcal{D}^b(\mmod \Phi)$ be a wide subcategory satisfying the following two conditions:
\begin{enumerate}
    \item\label{property1} $\mathcal{U}^\perp$ is functorially finite in $\mathcal{D}^b(\mmod \Phi)$,
    \item\label{property2} each object in $\mathcal{U}^\perp$ has its $\overline{\mathcal{F}}$-cover and its $\overline{\mathcal{F}}$-envelope in $\mathcal{U}^\perp\cap\overline{\mathcal{F}}$.
\end{enumerate}
A \textit{universal localization of $(\Phi, \mathcal{F})$ with respect to $\mathcal{U}$} is an algebra morphism $\Phi\xrightarrow{\phi}\Gamma$ initial in the category of algebra morphisms $\Phi\rightarrow\Lambda$ such that $\mathcal{U}\otimes_{\Phi}^{L}\Lambda=0$.

We prove that the above definition is in fact a generalisation and the base case recovers the definition of classic universal localization. Note that for $d=1$ property (\ref{property2}) in the above definition is automatically satisfied since $\overline{\mathcal{F}}=\mathcal{D}^b(\mmod\Phi)$ and so $\mathcal{U}^\perp\cap\overline{\mathcal{F}}=\mathcal{U}^\perp$.

\textbf{Proposition A (=Proposition \ref{remark_basecase_coincides}).}
In the classic case, that is $d=1$ and $\Phi$ is a finite dimensional hereditary $k$-algebra, a universal localization in the sense of Definition \ref{defn_higher_univ_loc} is precisely a universal localization in the classic sense. In particular, the objects of $\mathcal{U}$ are direct sums of shifts of $\Phi$-modules and $\mathcal{U}=\mathcal{R}^c_\mathcal{S}$, where $\mathcal{S}$ is the set of projective resolutions of these modules.

In the classic theory, the universal localization of $\Phi$ with respect to any given set $\mathcal{S}$ exists, even without any restriction on the global dimension of $\Phi$. Here it is not immediately clear whether a universal localization of $(\Phi, \mathcal{F})$ with respect to a given $\mathcal{U}$ satisfying the required properties exists. We show that in our setup it does, see Remark \ref{prop_univ_loc_exists}.

We prove the generalisation of Theorem \ref{thmKS}. In the following, a wide subcategory of $\mathcal{F}$ is an additive subcategory closed under $d$-kernels, $d$-cokernels and $d$-extensions and a wide subcategory of $\overline{\mathcal{F}}$ is an additive subcategory closed under $d$-extensions and $\Sigma^{\pm d}$, see Definitions \ref{defn_wide_dab} and \ref{defn_wide_angulated}. Moreover, a homological epimorphism of $d$-homological pairs $(\Phi, \mathcal{F})\xrightarrow{\phi}(\Gamma, \mathcal{G})$ is a homological epimorphism of $k$-algebras $\phi:\Phi\rightarrow\Gamma$ with the extra property that $\phi_*(\mathcal{G})\subseteq \mathcal{F}$, where $\phi_*$ is the restriction of scalars functor, see Definition \ref{defn_homo_epi_pairs}.

\textbf{Theorem B (= Theorem \ref{thm}).}
{\em
In the situation of Setup \ref{setup}, there are bijections between  the following sets.
\begin{enumerate}[label=(\alph*)]
    \item\label{itema} Functorially finite wide subcategories of $\mathcal{F}$.
    \item\label{itemb} Functorially finite wide subcategories of $\overline{\mathcal{F}}$.
    \item\label{itemc} Equivalence classes of homological epimorphisms of $d$-homological pairs $(\Phi, \mathcal{F})\xrightarrow{\phi}(\Gamma, \mathcal{G})$.
    \item\label{itemd} Equivalence classes of universal localizations $\Phi\xrightarrow{\phi}\Gamma$ of $(\Phi, \mathcal{F})$.
\end{enumerate}
}

Observe that the generalisation is partial and the sets (5), (6) and (7) from Theorem \ref{thmC} do not have a higher analogue. The reason is that currently there is no developed theory of a ``higher'' Mod$\,\Phi$. If such a theory develops in the future, it might be possible to prove a complete generalisation of the theorem. 

The paper is organized as follows. Section \ref{section_wide} recalls the construction of $\mathcal{F}$, $\overline{\mathcal{F}}$ and their wide subcategories. Section \ref{section_abc} proves the bijection between sets \ref{itema}, \ref{itemb} and \ref{itemc} from Theorem B. Section \ref{section_homoepi} studies homological epimorphisms, their associated restriction of scalars functor and recalls the theory of classic universal localizations. Section \ref{section_univ_loc} introduces the definition of universal localizations of $d$-homological pairs $(\Phi,\mathcal{F})$ and studies their existence and their relation to homological epimorphisms of $d$-homological pairs. Section \ref{section_proof} proves Theorem B and finally Section \ref{section_example} illustrates the theorem in an example.

\section{The categories $\mathcal{F}$ and $\overline{\mathcal{F}}$ and their wide subcategories}\label{section_wide}
In this section, we recall some important definitions useful for the study of $\mathcal{F}$, $\overline{\mathcal{F}}$ and their wide subcategories.

We start by recalling the definition of functorially finite subcategories.
\begin{defn}
Let $\mathcal{A}$ be a category and $\mathcal{X}$ a full subcategory of $\mathcal{A}$. An \textit{$\mathcal{X}$-precover} of an object $A$ in $\mathcal{A}$ is a morphism of the form $\xi:X\rightarrow A$ with $X\in\mathcal{X}$ such that every morphism $\xi':X'\rightarrow A$ with $X'\in\mathcal{X}$ factorizes as:
\begin{align*}
    \xymatrix{
    X'\ar[rr]^{\xi'}\ar@{-->}[rd]_{\exists}&& A.\\
    &X\ar[ru]_\xi&}
\end{align*}
An \textit{$\mathcal{X}$-cover} of $A$ is an $\mathcal{X}$-precover $\xi:X\rightarrow A$ which is also a right minimal morphism, that is each morphism $\varphi: X\rightarrow X$ with $\xi\circ\varphi=\xi$ is an isomorphism.

A \textit{strong $\mathcal{X}$-cover} of $A$ is an $\mathcal{X}$-cover $\xi:X\rightarrow A$ with the unique extension property, that is every $\xi':X'\rightarrow A$ with $X'\in\mathcal{X}$ factorizes uniquely as:
\begin{align*}
    \xymatrix{
    X'\ar[rr]^{\xi'}\ar@{-->}[rd]_{\exists !}&& A.\\
    &X\ar[ru]_\xi&}
\end{align*}

A full subcategory $\mathcal{X}$ of $\mathcal{A}$ is said to be \textit{precovering} if every object in $\mathcal{A}$ has an $\mathcal{X}$-precover.

The dual notions of precovers, covers, strong covers and precovering subcategories are \textit{preenvelopes}, \textit{envelopes}, \textit{strong envelopes} and \textit{preenveloping subcategories}. If a subcategory is both precovering and preenveloping, it is called \textit{functorially finite}.
\end{defn}

We now recall the definition of $d$-cluster tilting subcategories.

\begin{defn}[{\cite[Definition 2.2]{I2} and \cite[Definition 3.14]{GJ}}]\label{defn_dct}
Let $\mathcal{A}$ be either $\mmod\Phi$ or $\mathcal{D}^b(\mmod\Phi)$. A \textit{$d$-cluster tilting subcategory $\mathcal{C}$} of $\mathcal{A}$ is a functorially finite full subcategory such that
\begin{align*}
    \mathcal{C}=\{C\in\mathcal{A}\mid \Ext^{1 \dots d-1}_\mathcal{A}(\mathcal{A},C)=0\}=\{C\in\mathcal{A}\mid \Ext^{1 \dots d-1}_\mathcal{A}(C,\mathcal{A})=0\}.
\end{align*}
\end{defn}

As mentioned in the introduction, if there is a $d$-cluster tilting subcategory $\mathcal{F}\subseteq\mmod\Phi$, then $\mathcal{F}$ is a $d$-abelian category and $\overline{\mathcal{F}}$ is a $(d+2)$-angulated category. We recall the definitions of $d$-abelian and $(d+2)$-angulated categories.

\begin{defn}[{\cite[Definition 3.1]{GJ}}]\label{defn_dabelian}
		Let $\mathcal{A}$ be an additive category. A diagram of the form
			$\xymatrix{
				A^0\ar[r]& A^1\ar[r]& A^2\ar[r]&\cdots\ar[r]&A^{d-1}\ar[r]&A^d
			}$
			is a \textit{$d$-kernel} of a morphism $\xymatrix{A^d\ar[r]& A^{d+1}}$ if 
			\begin{align*}
			\xymatrix{
				0\ar[r] & A^0\ar[r]& A^1\ar[r]& A^2\ar[r]&\cdots\ar[r]&A^{d-1}\ar[r]&A^d\ar[r]& A^{d+1}
			}
			\end{align*}
			becomes an exact sequence under $\Hom_{\mathcal{A}}(B,-)$ for each $B$ in $\mathcal{A}$. The notion of \textit{$d$-cokernel} is dual.

A $d$-exact sequence is a diagram of the form
			\begin{align*}
			\xymatrix{
				0\ar[r]&A^0\ar[r]^{\alpha^0}& A^1\ar[r]& A^2\ar[r]&\cdots\ar[r]&A^{d-1}\ar[r]&A^d\ar[r]^{\alpha^d}& A^{d+1}\ar[r]&0,
			}
			\end{align*}
			such that $\xymatrix{A^0\ar[r]^{\alpha^0}& A^1\ar[r]& A^2\ar[r]&\cdots\ar[r]&A^{d-1}\ar[r]&A^d}$ is a $d$-kernel of $\alpha^d$ and 
			
			$\xymatrix{A^1\ar[r]& A^2\ar[r]&\cdots\ar[r]&A^{d-1}\ar[r]&A^d\ar[r]^{\alpha^d}& A^{d+1}}$ is a $d$-cokernel of $\alpha^0$.

A \textit{$d$-abelian category} is an additive category $\mathcal{A}$ which satisfies the following axioms:
		\begin{enumerate}
			\item[(A0)] The category $\mathcal{A}$ has split idempotents.
			\item[(A1)] Each morphism in $\mathcal{A}$ has a $d$-kernel and a $d$-cokernel.
			\item[(A2)] If $\alpha^0:\xymatrix{A^0\ar[r]&A^1}$ is a monomorphism and
			$\xymatrix{
				A^1\ar[r]&A^2\ar[r]&\cdots\ar[r]& A^{d+1}
			}$
			is a $d$-cokernel of $\alpha^0$, then
			\begin{align*}
			\xymatrix{
				0\ar[r]&A^0\ar[r]^{\alpha^0}& A^1\ar[r]& A^2\ar[r]&\cdots\ar[r]&A^{d-1}\ar[r]&A^d\ar[r]& A^{d+1}\ar[r]&0,
			}
			\end{align*}
			is a $d$-exact sequence.
			\item[(A2$^{\text{op}}$)] If $\alpha^d:\xymatrix{A^d\ar[r]&A^{d+1}}$ is an epimorphism and
			$\xymatrix{
				A^0\ar[r]&\cdots\ar[r]& A^{d-1}\ar[r]& A^{d}
			}$
			is a $d$-kernel of $\alpha^d$, then
			\begin{align*}
			\xymatrix{
				0\ar[r]&A^0\ar[r]& A^1\ar[r]& A^2\ar[r]&\cdots\ar[r]&A^{d-1}\ar[r]&A^d\ar[r]^{\alpha^d}& A^{d+1}\ar[r]&0,
			}
			\end{align*}
			is a $d$-exact sequence.
		\end{enumerate}
	\end{defn}

	\begin{defn}[{\cite[Definition\ 1.1]{GKO}}]\label{defn_angulated}
		Let $\mathcal{M}$ be an additive category and let $\Sigma^d$ be an automorphism of $\mathcal{M}$ with inverse $\Sigma^{-d}$. A \textit{$\Sigma^d$-sequence} is a sequence of morphisms in $\mathcal{M}$ of the form
		\begin{align}\tag{$\star$}\label{diagram_angle}
		\xymatrix {
			\epsilon : & X^0\ar[r]^{\xi^0} & X^1\ar[r]^{\xi^1}& X^2\ar[r] &\cdots\ar[r] & X^d\ar[r]^{\xi^d} & X^{d+1}\ar[r]^{\xi^{d+1}} &\Sigma^d X^0.
		}
		\end{align}
		A \textit{morphism of $\Sigma^d$-sequences} is given by a sequence of morphisms $\varphi=(\varphi^0,\dots,\,\varphi^{d+1})$ such that the following diagram commutes:
		\begin{align*}
		\xymatrix {
			\epsilon:\ar[d]^\varphi &X^0\ar[r]^{\xi^0}\ar[d]^{\varphi^0} & X^1\ar[r]^{\xi^1}\ar[d]^{\varphi^1}& X^2\ar[r]\ar[d]^{\varphi^2} &\cdots\ar[r] & X^d\ar[r]^{\xi^d}\ar[d]^{\varphi^d} & X^{d+1}\ar[r]^{\xi^{d+1}}\ar[d]^{\varphi^{d+1}} &\Sigma^d X^0\ar[d]^{\Sigma^d \varphi^0}\\
			\epsilon': &Y^0\ar[r]_{\eta^0} & Y^1\ar[r]_{\eta^1}& Y^2\ar[r] &\cdots\ar[r] & Y^d\ar[r]_{\eta^d} & Y^{d+1}\ar[r]_{\eta^{d+1}} &\Sigma^d Y^0.
		}
		\end{align*}
		A $(d+2)$\textit{-angulated category} is a triple $(\mathcal{M},\Sigma^d,\pentagon)$, where $\mathcal{M}$, $\Sigma^d$ are as above and $\pentagon$ is a collection of $\Sigma^d$-sequences, called \textit{$(d+2)$-angles}, satisfying the following axioms.
		\begin{enumerate}
			\item[(N1)] The collection $\pentagon$ is closed under sums and summands and, for every $X\in\mathcal{M}$, the \textit{trivial $\Sigma^d$-sequence}
			\begin{align*}
			\xymatrix {
				\epsilon : & X\ar[r]^{1_{X}} & X\ar[r]& 0\ar[r] &\cdots\ar[r] & 0\ar[r] & 0\ar[r] &\Sigma^d X}
			\end{align*}
			is in $\pentagon$.
			For each morphism $\xi^0:X^0\rightarrow X^1$ in $\mathcal{M}$, there is a $(d+2)$-angle in $\pentagon$ of the form (\ref{diagram_angle}).
			
			\item[(N2)] A $\Sigma^d$-sequence (\ref{diagram_angle}) is in $\pentagon$ if and only if so is its \textit{left rotation}:
			\begin{align*}
			\xymatrix {
				X^1\ar[r]^{\xi^1}& X^2\ar[r] &\cdots\ar[r] & X^d\ar[r]^{\xi^d} & X^{d+1}\ar[r]^{\xi^{d+1}} &\Sigma^d X^0\ar[rr]^{(-1)^d\Sigma^d(\xi^0)} && \Sigma^d X^1.
			}
			\end{align*}
			
			\item[(N3)] Each commutative diagram of solid arrows, with rows in $\pentagon$
			\begin{align*}
			\xymatrix {
				X^0\ar[r]^{\xi^0}\ar[d]^{\varphi^0} & X^1\ar[r]^{\xi^1}\ar[d]^{\varphi^1}& X^2\ar[r]\ar@{-->}[d]^{\varphi^2} &\cdots\ar[r] & X^d\ar[r]^{\xi^d}\ar@{-->}[d]^{\varphi^d} & X^{d+1}\ar[r]^{\xi^{d+1}}\ar@{-->}[d]^{\varphi^{d+1}} &\Sigma^d X^0\ar[d]^{\Sigma^d \varphi^0}\\
				Y^0\ar[r]_{\eta^0} & Y^1\ar[r]_{\eta^1}& Y^2\ar[r] &\cdots\ar[r] & Y^d\ar[r]_{\eta^d} & Y^{d+1}\ar[r]_{\eta^{d+1}} &\Sigma^d Y^0,
			}
			\end{align*}
			can be completed as indicated to a morphism of $\Sigma^d$-sequences.
			
			\item[(N4)] In the situation of (N3), the morphisms $\varphi^2,\dots,\,\varphi^{d+1}$ can be chosen such that
			\begin{align*}
			\xymatrix{
				X^1\oplus Y^0\ar[rr]^{\begin{psmallmatrix} -\xi^1 &0\\ \varphi^1 &\eta^0 \end{psmallmatrix}} &&X^2\oplus Y^1\ar[r]& \cdots\ar[r]& \Sigma^d X^0\oplus Y^{d+1}\ar[rr]^{\begin{psmallmatrix} -\Sigma^d\xi^0 &0\\ \Sigma\varphi^0 &\eta^{d+1} \end{psmallmatrix}}&& \Sigma^d X^1 \oplus \Sigma^d Y^0
			}
			\end{align*}
			belongs to $\pentagon$.
		\end{enumerate}
	\end{defn}

We recall the definition of additive subcategories of additive categories and in particular wide subcategories of $d$-abelian and $(d+2)$-angulated categories.

\begin{defn}
Let $\mathcal{A}$ be an additive category. An \textit{additive subcategory} of $\mathcal{A}$ is a full subcategory which is closed under direct sums, direct summands and isomorphisms in $\mathcal{A}$.
\end{defn}

\begin{defn}[{\cite[Definition 2.11]{HJV}}]\label{defn_wide_dab}
Let $\mathcal{A}$ be a $d$-abelian category. An additive subcategory $\mathcal{W}$ of $\mathcal{A}$ is called \textit{wide} if it satisfies the following conditions.
\begin{enumerate}[label=(\alph*)]
    \item Each morphism in $\mathcal{W}$ has a $d$-kernel in $\mathcal{A}$ consisting of objects in $\mathcal{W}$.
    \item Each morphism in $\mathcal{W}$ has a $d$-cokernel in $\mathcal{A}$ consisting of objects in $\mathcal{W}$.
    \item Each $d$-exact sequence in $\mathcal{A}$ of the form
    \begin{align*}
        0\rightarrow w'\rightarrow a_d\rightarrow\cdots\rightarrow a_1\rightarrow w\rightarrow 0
    \end{align*}
    with $w,w'\in\mathcal{W}$ is Yoneda equivalent in the sense of \cite[Chapter IV.9]{HS} to a $d$-exact sequence in $\mathcal{A}$ of the form
    \begin{align*}
        0\rightarrow w'\rightarrow w_d\rightarrow\cdots\rightarrow w_1\rightarrow w\rightarrow 0,
    \end{align*}
    with $w_i\in\mathcal{W}$ for each $i$.
\end{enumerate}
\end{defn}

\begin{defn}\label{defn_wide_angulated}
Let $\mathcal{A}$ be a $(d+2)$-angulated category. An additive subcategory $\mathcal{W}$ of $\mathcal{A}$ is called \textit{wide} if it satisfies the following conditions.
\begin{enumerate}[label=(\alph*)]
    \item For any morphism in $\mathcal{A}$ of the form $\delta: w\rightarrow \Sigma^d w'$, with $w,\,w'\in\mathcal{W}$ there is a $(d+2)$-angle in $\mathcal{A}$ of the form
    \begin{align*}
        w'\rightarrow w_d\rightarrow\cdots\rightarrow w_1\rightarrow w\xrightarrow{\delta}\Sigma^d w',
    \end{align*}
with $w_i\in\mathcal{W}$ for each $i$.
\item $\Sigma^d(\mathcal{W})\subseteq \mathcal{W}$ and $\Sigma^{-d}(\mathcal{W})\subseteq \mathcal{W}$.
\end{enumerate}
\end{defn}

\section{The bijective correspondences between three of the four sets}\label{section_abc}
In this section, we prove the existence of the bijections between sets \ref{itema}, \ref{itemb} and \ref{itemc} in Theorem \ref{thm}.

The following proposition is a restatement of a previous result and proves the bijection between \ref{itema} and \ref{itemb} in Theorem \ref{thm}.

\begin{proposition}[{\cite[Theorem A]{FF}}]\label{lemma_ab}
There is a bijection
\begin{align*}
    \Bigg\{
    \begin{matrix} \text{functorially finite}\\ \text{wide subcategories} \\ \text{of } \mathcal{F} \end{matrix}
    \Bigg\}
    \rightarrow
    \Bigg\{
    \begin{matrix} \text{functorially finite}\\ \text{wide subcategories} \\ \text{of } \overline{\mathcal{F}} \end{matrix}
    \Bigg\}
    \end{align*}
    sending a functorially finite wide subcategory $\mathcal{W}$ of $\mathcal{F}$ to $\overline{\mathcal{W}}:=\text{add}\{\Sigma^{id}\mathcal{W}\mid i\in\mathbb{Z}\}$.
\end{proposition}

%\section{Homological epimorphisms of $d$-homological pairs}\label{section_homo_epi}
In order to study \ref{itemc} in Theorem \ref{thm}, we first recall the restriction of scalars functor and the definitions of some special epimorphisms  of $k$-algebras.

\begin{remark}
Let $\phi:\Phi\rightarrow \Gamma$ be an algebra homomorphism. The restriction of scalars functor
\begin{align*}
    \phi_*(-)\cong\Hom_\Gamma(\Gamma, -)\cong -\otimes_\Gamma \Gamma: \mmod\Gamma\rightarrow \mmod\Phi
\end{align*}
is a faithful, exact functor. We get the diagram
\begin{align*}
    \xymatrix @C=8em{
    \mmod\Gamma \ar[r]^{\phi_*(-)}& \mmod \Phi, \ar@/^2pc/[l]^{\phi^!(-)} \ar@/_2pc/[l]_{\phi^*(-)}
    }
\end{align*}
where $\phi^*( - )=-\otimes_\Phi \Gamma$, $\phi^!(-)=\Hom_\Phi(\Gamma,-)$ and each functor is the left adjoint functor of the functor drawn below it.
\end{remark}

\begin{defn}\label{defn_homo_epi_pairs}
An \textit{epimorphism of $k$-algebras} $\Phi\xrightarrow{\phi} \Gamma$ is an algebra homomorphism which is an epimorphism in the category of rings. Equivalently, it is an algebra homomorphism such that the multiplication map $\Gamma\otimes_{\Phi} \Gamma\rightarrow \Gamma$ is an isomorphism, see \cite[Proposition 1.1]{SL}.

A \textit{homological epimorphism of $k$-algebras} $\Phi\xrightarrow{\phi} \Gamma$ is an algebra homomorphism such that the multiplication map $\Gamma\otimes_{\Phi} \Gamma\rightarrow \Gamma$ is an isomorphism and Tor$^\Phi_i(\Gamma, \Gamma)=0$ for all $i>0$, see \cite[Definition 4.5]{GL}. In particular, note that it is an epimorphism of $k$-algebras.

An \textit{epimorphism of $d$-homological pairs} $(\Phi,\mathcal{F})\xrightarrow{\phi}(\Gamma, \mathcal{G})$ is an epimorphism of $k$-algebras $\Phi\xrightarrow{\phi} \Gamma$ such that $\phi_*(\mathcal{G})\subseteq \mathcal{F}$. We say that $\phi$ is \textit{$d$-pseudoflat} if moreover Tor$^\Phi_d(\Gamma,\Gamma)=0$.

Similarly, a \textit{homological epimorphism of $d$-homological pairs} $(\Phi,\mathcal{F})\xrightarrow{\phi}(\Gamma, \mathcal{G})$ is a homological epimorphism of $k$-algebras $\Phi\xrightarrow{\phi} \Gamma$ such that $\phi_*(\mathcal{G})\subseteq \mathcal{F}$.
\end{defn}

We now prove the bijection between \ref{itema} and \ref{itemc} in Theorem \ref{thm}.

\begin{proposition}\label{lemma_ac}
There is a bijection 
\begin{align*}
   \Bigg\{
    \begin{matrix} \text{equivalence classes of homological}\\ \text{ epimorphisms of $d$-homological pairs} \\ (\Phi, \mathcal{F})\xrightarrow{\phi}(\Gamma, \mathcal{G}) \end{matrix}
    \Bigg\} 
    \rightarrow
    \Bigg\{
    \begin{matrix} \text{functorially finite}\\ \text{wide subcategories} \\ \text{of } \mathcal{F} \end{matrix}
    \Bigg\}
    \end{align*}
sending the equivalence class of 
 $(\Phi, \mathcal{F})\xrightarrow{\phi}(\Gamma, \mathcal{G})$ to $\phi_*(\mathcal{G})$.
\end{proposition}

\begin{proof}
By \cite[Theorem A]{HJV}, there is a bijection between the set of equivalence classes of $d$-pseudoflat epimorphisms of $d$-homological pairs $(\Phi, \mathcal{F})\xrightarrow{\phi}(\Gamma, \mathcal{G})$ and the set of functorially finite wide subcategories of $\mathcal{F}$ sending the equivalence class of 
 $(\Phi, \mathcal{F})\xrightarrow{\phi}(\Gamma, \mathcal{G})$ to $\phi_*(\mathcal{G})$. Hence, it is enough to prove that a morphism of $d$-homological pairs $(\Phi, \mathcal{F})\xrightarrow{\phi}(\Gamma, \mathcal{G})$ is $d$-pseudoflat if and only if it is a homological epimorphism.

Recall that gldim$\Phi\leq d$. By \cite[Proposition 5.8(i)]{HJV} we have that if $(\Phi, \mathcal{F})\xrightarrow{\phi}(\Gamma, \mathcal{G})$ is a $d$-pseudoflat epimorphism of $d$-homological pairs, then it is also a homological epimorphism of $d$-homological pairs.
Suppose now that $(\Phi, \mathcal{F})\xrightarrow{\phi}(\Gamma, \mathcal{G})$ is a homological epimorphism of $d$-homological pairs. Then, $\Gamma\otimes_{\Phi} \Gamma\cong \Gamma$ and Tor$^\Phi_i(\Gamma, \Gamma)=0$ for all $i>0$ by Definition \ref{defn_homo_epi_pairs}. In particular, Tor$^\Phi_d(\Gamma, \Gamma)=0$ and $\phi$ is a $d$-pseudoflat epimorphism of $d$-homological pairs.
\end{proof}

\begin{remark}\label{remark_construction6.3}
In the proof above, we showed that in our setup homological epimorphisms of $d$-homological pairs coincide with $d$-pseudoflat epimorphisms of $d$-homological pairs. Hence, the inverse of the bijection from Proposition \ref{lemma_ac} sends a functorially finite wide subcategory $\mathcal{W}$ of $\mathcal{F}$ to a homological epimorphisms of $d$-homological pairs constructed as follows, see \cite[Construction 6.3 and Proposition 6.4]{HJV} for more details.

Let $\iota^*:\mathcal{F}\rightarrow \mathcal{W}$ be the left adjoint to the inclusion $\iota_*:\mathcal{W}\rightarrow \mathcal{F}$. Set 
\begin{align*}
   s=\iota^*(\Phi_\Phi) \text{ and } \Gamma=\End_\Phi (s).
\end{align*}
There is an algebra homomorphism $\phi$ given by the composition
\begin{align*}
    \Phi\xrightarrow{\sim} \End_\Phi(\Phi)\xrightarrow{\iota^*(-)} \End_\Phi (s)=\Gamma.
\end{align*}
Then, letting $\mathcal{G}:=\phi^{!}(\mathcal{W})\subseteq \mmod\Gamma$, we have that
\begin{align*}
    \phi: (\Phi,\mathcal{F})\rightarrow (\Gamma, \mathcal{G})
\end{align*}
is a homological epimorphism of $d$-homological pairs with $\phi_*(\mathcal{G})=\mathcal{W}$.
\end{remark}

\section{Homological epimorphisms, restriction of scalars and universal localizations}\label{section_homoepi}

%\textcolor{red}{An alternative: I could swap Sections 3 and 4 and include in this section the diagram of restriction of scalars and the definition of epimorphism and homological epimorphism of algebras (currently Remark 3.2 and half of Definition 3.3). Pro: the non-higher theory would be in one block in this section and before we pass to higher theory. Con: this would move the part on classic universal localizations further from the section on its generalisation. In theory, this non-higher section could also be moved before current section 2 to fully have the non-higher before higher (some slight changes would have to be made, like bringing the definition of (strong) covers/envelopes into this section). But I'm not sure if this would be a good idea.}

In this section, we study homological epimorphisms and their restriction of scalars functors. We then recall the classic notion of universal localizations and study an equivalent definition.
Note that this survey does not take place in higher homological algebra, but it will be useful to understand the generalisation of universal localizations we will introduce in the next section.
%When $\phi:\Phi\rightarrow \Gamma$ is an epimorphism of $k$-algebras, we can say more about the restriction of scalars functor.

\begin{remark}\label{remark_fullfaithful}
If $\phi:\Phi\rightarrow \Gamma$ is an epimorphism of $k$-algebras, then $\phi_* (-)$ is a faithful, full exact functor, see \cite[Remark 5.2]{HJV}. If we further assume that $\phi$ is a homological epimorphism, then
\begin{align*}
   \phi_*(-)=R\Hom_\Gamma(\Gamma, -): \mathcal{D}^b(\mmod\Gamma)\rightarrow \mathcal{D}^b(\mmod\Phi)
\end{align*}
is a full and faithful triangulated functor, see \cite[Theorem 4.4]{GL} and note that the argument goes through also substituting Mod with mod. Moreover, $\phi_*(-)$ has left adjoint functor
\begin{align*}
   \phi^*(-)=-\otimes_\Phi^L\Gamma: \mathcal{D}^b(\mmod\Phi)\rightarrow \mathcal{D}^b(\mmod\Gamma)
\end{align*}
and right adjoint functor
\begin{align*}
   \phi^!(-)=R\Hom_\Phi(\Gamma, -): \mathcal{D}^b(\mmod\Phi)\rightarrow \mathcal{D}^b(\mmod\Gamma).
\end{align*}
%THIS IS FOR THE LEFT ADJOINT: see \cite[Lemma 5.9]{HJV}, where we can restrict to bounded derived categories of finitely generated modules because gldim$\Phi\leq d$ and so gldim$\Gamma\leq d$ by \cite[Proposition 5.8]{HJV}.
%I DON'T HAVE A REFERENCE FOR THE RIGHT ADJOINT ONE!!
\end{remark}

In this section, we work in the following setup.

\begin{setup}
Let $\phi:\Phi\rightarrow \Gamma$ be a homological epimorphism and let 
\begin{align*}
    \mathcal{U}= {}^\perp(\phi_*(\mathcal{D}^b(\mmod\Gamma))).
\end{align*}
\end{setup}

\begin{lemma}\label{lemma_U_wide}
The subcategory $\mathcal{U}$ is a wide subcategory of $\mathcal{D}^b(\mmod\Phi)$.
\end{lemma}

\begin{proof}
It is easy to see that $\mathcal{U}$ is closed under direct summands. We show it is closed under $\Sigma^{\pm 1}$ and under extensions.

Note that $\phi_*(\mathcal{D}^b(\mmod\Gamma))$ is a triangulated subcategory of $\mathcal{D}^b(\mmod\Phi))$ because
\begin{align*}
    \phi_*:\mathcal{D}^b(\mmod\Gamma)\rightarrow\mathcal{D}^b(\mmod\Phi)
\end{align*}
is a full and faithful triangulated functor by Remark \ref{remark_fullfaithful}. Then, any $x\in \phi_*(\mathcal{D}^b(\mmod\Gamma))$ is such that $\Sigma^{-i}x\in \phi_*(\mathcal{D}^b(\mmod\Gamma))$ and so for any $u\in\mathcal{U}$ and integer $i$, we have that
\begin{align*}
    \Hom_{\mathcal{D}^b(\mmod\Phi)}(\Sigma^i u, x)\cong 
    \Hom_{\mathcal{D}^b(\mmod\Phi)}(u, \Sigma^{-i} x)=0.
\end{align*}
Hence $\Sigma^iu\in\mathcal{U}$ and $\mathcal{U}$ is closed under $\Sigma^{\pm 1}$.

Consider now a triangle in $\mathcal{D}^b(\mmod\Phi)$ of the form
\begin{align*}
    u\rightarrow t\rightarrow v\rightarrow \Sigma u,
\end{align*}
with $u$ and $v$ in $\mathcal{U}$. For any $x\in\phi_*(\mathcal{D}^b(\mmod\Gamma))$, we have a long exact sequence of the form
\begin{align*}
   \cdots\rightarrow \Hom_{\mathcal{D}^b(\mmod\Phi)}(v,x)
   \rightarrow \Hom_{\mathcal{D}^b(\mmod\Phi)}(t,x)
   \rightarrow \Hom_{\mathcal{D}^b(\mmod\Phi)}(u,x)\rightarrow\cdots.
\end{align*}
Note that the first and the last term of the above sequence are zero, then so is the second one and we conclude that $t\in\mathcal{U}$.
\end{proof}

\begin{lemma}\label{lemma_tstructure}
The pair $(\mathcal{U},\phi_*(\mathcal{D}^b (\mmod\Gamma)))$ is a stable $t$-structure in the sense of \cite[Definition 9.14]{M}.
\end{lemma}
\begin{proof}
 Since both $\mathcal{U}$ and $\phi_*(\mathcal{D}^b (\mmod\Gamma))$ are triangulated subcategories and $\mathcal{U}= {}^\perp(\phi_*(\mathcal{D}^b(\mmod\Gamma)))$, it is enough to prove that $\mathcal{U}*\phi_*(\mathcal{D}^b(\mmod\Gamma))=\mathcal{D}^b(\mmod\Phi)$.

By \cite[Lemma 4.3]{HJV} applied to the pair of adjoint functors from Remark \ref{remark_fullfaithful}
\begin{align*}
    \xymatrix@C=3em{
    \mathcal{D}^b(\mmod\Gamma)\ar@/^/[r]^{\phi^*} & \mathcal{D}^b(\mmod\Phi)\ar@/^/[l]^{\phi_*}
    }
\end{align*}
we have that $\phi_*(\mathcal{D}^b(\mmod\Gamma))\subseteq \mathcal{D}^b(\mmod\Phi)$ is strongly enveloping. In particular, any object $x\in \mathcal{D}^b(\mmod\Phi)$ has strong $\phi_*(\mathcal{D}^b(\mmod\Gamma))$-envelope $\eta_x:x\rightarrow\phi_*\phi^*(x)$, where $\eta$ is the unit of the adjunction. Extend $\eta_x$ to a triangle in $\mathcal{D}^b(\mmod\Phi)$:
\begin{align*}
    y\rightarrow x \xrightarrow{\eta_x}\phi_*\phi^*(x)\rightarrow \Sigma y. 
\end{align*}
Since $\phi_*(\mathcal{D}^b(\mmod\Gamma))\subseteq \mathcal{D}^b(\mmod\Phi)$ is full, closed under extensions and direct summands, by the dual of \cite[Lemma 2.1]{PJ}, we have that 
\begin{align*}
    \Hom_{\mathcal{D}^b(\mmod\Phi)}(y,\phi_*(\mathcal{D}^b(\mmod\Gamma)))=0,
\end{align*}
that is $y\in\mathcal{U}$.
\end{proof}

\begin{lemma}\label{lemma_initial_prop}
The morphism $\phi:\Phi\rightarrow \Gamma$ is initial in the category of algebra morphisms $\Phi\rightarrow \Lambda$ with the property that $\mathcal{U}\otimes^L_\Phi \Lambda=0$.
\end{lemma}
\begin{proof}
Since $\phi^*$ is the left adjoint of $\phi_*$, we have that
\begin{align*}
    0=\Hom_{\mathcal{D}^b(\mmod\Phi)}(\mathcal{U},\phi_*(\mathcal{D}^b(\mmod\Gamma)))\cong
    \Hom_{\mathcal{D}^b(\mmod\Gamma)}(\phi^*(\mathcal{U}),\mathcal{D}^b(\mmod\Gamma)),
\end{align*}

and so $\phi^*(\mathcal{U})=\mathcal{U}\otimes^L_\Phi \Gamma=0$.
Suppose $\Lambda$ is another $k$-algebra with the property that $\mathcal{U}\otimes^L_\Phi \Lambda=0$. There is a commutative diagram of functors
\begin{align*}
    \xymatrix@C=3em{
    \mathcal{U}\ar@{^{(}->}[r]^-{i}\ar[rd]_0& \mathcal{D}^b(\mmod\Phi)\ar[d]^f & \phi_*(\mathcal{D}^b(\mmod\Gamma))\ar@{_{(}->}[l]_{j}\ar[ld]^{g}\\
    & \mathcal{D}^b(\mmod\Lambda) &
    }
\end{align*}
where $f:=-\otimes^L_\Phi \Lambda$ and $g:=f\circ j$. Let $q:\mathcal{D}^b(\mmod\Phi)\rightarrow\phi_*(\mathcal{D}^b(\mmod\Gamma))$ be the functor obtained applying $\phi_*\phi^*(-)$. In particular, $q$ sends each object $x\in \mathcal{D}^b(\mmod\Phi)$ to its strong $\phi_*(\mathcal{D}^b(\mmod\Gamma))$-envelope $\phi_*\phi^*(x)$. 

We now prove that there is a natural equivalence $\zeta: f\xrightarrow{\sim} g\circ q$. Applying the functor $f$ to the natural transformation $\eta: 1_{\mathcal{D}^b(\mmod\Phi)}\rightarrow j\circ q$ given by strong $\phi_*(\mathcal{D}^b(\mmod\Gamma))$-envelopes  $\eta_x:x\rightarrow\phi_*\phi^*(x)$, we obtain a natural transformation $\zeta: f\rightarrow f\circ j\circ q= g\circ q$. Recall that for any $x\in \mathcal{D}^b(\mmod\Phi)$, the triangle
\begin{align*}
    y\rightarrow x \xrightarrow{\eta_x}\phi_*\phi^*(x)\rightarrow \Sigma y 
\end{align*}
is such that $y\in\mathcal{U}$. Then $f(y)=0$ by construction and $f(\eta_x):f(x)\xrightarrow{\sim}g\circ q(x)$ is an isomorphism.

Viewing $\Phi$ as an element of $\mathcal{D}^b(\mmod\Phi)$, we have that
\begin{align*}
    & q(\Phi)=\phi_*\phi^*(\Phi)=R\Hom_\Gamma(\Gamma, \Phi\otimes^L_\Phi \Gamma)\cong \Gamma,\\
    & g\circ q(\Phi)\cong g(\Gamma), \\
    & f(\Phi)=\Phi\otimes^L_\Phi \Lambda\cong \Lambda.
\end{align*}

We prove that the following diagram commutes
\begin{align}\label{diagram_commutative}
\xymatrix@C=3em{    
    \End_{\mathcal{D}^b(\mmod\Phi)} (\Phi)\ar[r]^-{q(-)}\ar[d]^{f(-)}& \End_{\phi_*(\mathcal{D}^b(\mmod\Gamma))}(\Gamma)\ar[d]^{g(-)}\\
    \End_{\mathcal{D}^b(\mmod\Lambda)}(\Lambda)\ar[r]^-e_-\sim &\End_{\mathcal{D}^b(\mmod\Lambda)}(g(\Gamma)),
}
\end{align}
where $e$ is the isomorphism of rings induced by the isomorphism 
\begin{align*}
    \zeta_\Phi:\Lambda\cong f(\Phi)\xrightarrow{\sim}g\circ q(\Phi)\cong g(\Gamma).
\end{align*}
For any element $\lambda\cdot:\Lambda\rightarrow\Lambda$ in $\End_{\mathcal{D}^b(\mmod\Lambda)}(\Lambda)$, there is a commutative diagram:
\begin{align*}
    \xymatrix@C=3em{
    g(\Gamma)\ar[r]^{e(\lambda\cdot)}\ar[d]_{\zeta^{-1}_\Phi}& g(\Gamma)\\
    \Lambda\ar[r]_{\lambda\cdot}&\Lambda\ar[u]_{\zeta_\Phi}.
    }
\end{align*}
Similarly, for any $\varphi\cdot\in\End_{\mathcal{D}^b(\mmod\Phi)}(\Phi)$, there is a commutative diagram:
\begin{align*}
    \xymatrix@C=3em{
    g\circ q(\Phi)\ar[r]^{g\circ q(\varphi\cdot)}\ar[d]_{\zeta^{-1}_\Phi}& g\circ q(\Phi)\\
    \Lambda\ar[r]_{f(\varphi\cdot)}&\Lambda\ar[u]_{\zeta_\Phi}.
    }
\end{align*}
Note that for $\lambda\cdot=f(\varphi\cdot)$ the two diagrams coincide and hence $g\circ q(\varphi\cdot)=e(f(\varphi\cdot))$, that is diagram (\ref{diagram_commutative}) commutes.

Moreover, the diagram

\begin{align}\label{diagram_commutative2}
\xymatrix@C=3em @R=3em{    
    \Phi\ar[d]_{\rotatebox{90}{$\sim$}}\ar[r]^\phi&\Gamma\ar[d]^{\rotatebox{90}{$\sim$}}\\
    \End_{\mathcal{D}^b(\mmod\Phi)} (\Phi)\ar[r]^-{q(-)}& \End_{\phi_*(\mathcal{D}^b(\mmod\Gamma))}(\Gamma)
}
\end{align}

commutes and the vertical arrows are isomorphisms. Combining diagrams (\ref{diagram_commutative2}) and (\ref{diagram_commutative}), we obtain the commutative diagram

\begin{align*}
\xymatrix@C=3em@R=3em{    
    \Phi\ar[d]_{\rotatebox{90}{$\sim$}}\ar[r]^\phi&\Gamma\ar[d]^{\rotatebox{90}{$\sim$}}\\
    \End_{\mathcal{D}^b(\mmod\Phi)} (\Phi)\ar[r]^-{q(-)}\ar[d]^{f(-)}& \End_{\phi_*(\mathcal{D}^b(\mmod\Gamma))}(\Gamma)\ar[d]^{g(-)}\\
    \End_{\mathcal{D}^b(\mmod\Lambda)}(\Lambda)\ar[d]_{\rotatebox{90}{$\sim$}} &\End_{\mathcal{D}^b(\mmod\Lambda)}(g(\Gamma))\ar[l]^-{e^{-1}}_-\sim\\
    \Lambda.
}
\end{align*}
Hence we have the commutative diagram
\begin{align*}
    \xymatrix{
    \Phi\ar[r]^{\phi}\ar[d]& \Gamma\ar[ld]^{\gamma}\\
    \Lambda
    }
\end{align*}
where $\gamma$ is unique because $\phi$ is a homological epimorphism. So we conclude that $\phi$ is initial in the category of algebra morphisms $\Phi\rightarrow \Lambda$ such that $\mathcal{U}\otimes^L_\Phi \Lambda=0$.
\end{proof}

\subsection{Universal localizations}
In the rest of this section, we recall and study the notion of universal localizations.

\begin{defn}\label{defn_classic}
A \textit{universal localization of $\Phi$} with respect to a set $\mathcal{S}$ of morphisms between finitely generated projective right $\Phi$-modules is an algebra homomorphism $\Phi\rightarrow \Phi_\mathcal{S}$ universal with respect to the property that every element $\sigma\otimes_\Phi \Phi_\mathcal{S}$ with $\sigma\in\mathcal{S}$ has an inverse.
\end{defn}
As mentioned in the introduction, by \cite[Theorem 4.1]{S}, the universal localization of $\Phi$ with respect to any such set $\mathcal{S}$ exists and it is unique up to unique isomorphism.

Alternatively, universal localizations can be defined using $\mathcal{D}^{\perf}(\Phi)$, that is the homotopy category whose objects are bounded complexes of finitely generated projective $\Phi$-modules, that is \textit{perfect complexes}, and whose morphisms are homotopy equivalence classes of chain maps.

\begin{remark}\label{remark_alternative_defn}
Let $\mathcal{S}$ be a set of maps between finitely generated projective right $\Phi$-modules. The set $\mathcal{S}$ can be viewed as a set of objects in $\mathcal{D}^{\perf}(\Phi)$, where each element of $\mathcal{S}$ is a complex concentrated in two degrees. Let $\mathcal{R}^c_\mathcal{S}$ be the smallest triangulated subcategory of $\mathcal{D}^{\perf}(\Phi)$ containing $\mathcal{S}$ and all the direct summands of any of its objects.
Then, for any morphism $\sigma\in\mathcal{S}$, we have that $\sigma\otimes^L_\Phi \Phi_\Lambda=0$ if and only if $\sigma\otimes_\Phi \Lambda$ is acyclic and hence invertible. By the way $\mathcal{R}^c_\mathcal{S}$ is constructed, we conclude that $\mathcal{R}^c_\mathcal{S}\otimes^L_\Phi \Lambda=0$ if and only if $\sigma\otimes_\Phi \Lambda$ is invertible for each $\sigma\in\mathcal{S}$. One can also easily see that the universal property from Definition \ref{defn_classic} is equivalent to the initial property described above.
%Moreover, let $\mathcal{T}^c$ be the idempotent completion of the Verdier quotient $\mathcal{D}^{b}(\mmod\Phi)/\mathcal{R}^c$.
\end{remark}

By the above remark, the universal localization of $\Phi$ with respect to $\mathcal{S}$ can alternatively be defined as follows.

\begin{defn}\label{defn_alternative}
Let $\mathcal{S}$ be a set of maps between finitely generated projective right $\Phi$-modules and $\mathcal{R}^c_\mathcal{S}$ be the smallest triangulated subcategory of $\mathcal{D}^{\perf}(\Phi)$ containing $\mathcal{S}$ and all the direct summands of any of its objects. The universal localization of $\Phi$ with respect to $\mathcal{S}$ is the algebra homomorphism $\Phi\rightarrow\Phi_\mathcal{S}$ initial in the category of $k$-algebra homomorphisms $\Phi\rightarrow\Lambda$ with the property that $\mathcal{R}^c_\mathcal{S}\otimes^L_\Phi \Lambda=0$.
\end{defn}

Neeman and Ranicki showed in \cite{NR} how this construction links universal localizations to Verdier quotients. In the following, let $\mathcal{T}^c$ be the idempotent completion of the Verdier quotient $\mathcal{D}^{\perf}(\Phi)/\mathcal{R}^c_\mathcal{S}$.

\begin{theorem}[{\cite[Theorem 0.7]{NR}}]\label{thmC}
 Consider the natural functor $\mathcal{D}^{\perf}(\Phi)\rightarrow \mathcal{D}^{\perf}(\Phi_\mathcal{S})$ which takes a complex $C$ in $\mathcal{D}^{\perf}(\Phi)$ and sends it to $C\otimes_\Phi\Phi_\mathcal{S}$. Take the canonical factorization
 \begin{align*}
     \mathcal{D}^{\perf}(\Phi)\xrightarrow{\pi} \mathcal{T}^c\xrightarrow{T}\mathcal{D}^{\perf}(\Phi_\mathcal{S}).
 \end{align*}
Then the following are equivalent:
\begin{enumerate}
    \item The functor $T$ is an equivalence of categories.
    \item For all $n>0$, $\Tor^\Phi_n(\Phi_\mathcal{S},\Phi_\mathcal{S})=0$.
\end{enumerate}
\end{theorem}

\begin{remark}\label{remarkD}
Note that in the case when $\Phi$ is hereditary, Theorem \ref{thmKS} implies that any universal localization $\Phi\rightarrow\Phi_\mathcal{S}$ is also a homological epimorphism (and viceversa). In particular, $\Tor^\Phi_n(\Phi_\mathcal{S},\Phi_\mathcal{S})=0$ for all $n>0$. Hence, by Theorem \ref{thmC}, there is an equivalence of categories
\begin{align*}
   T:\mathcal{T}^c\xrightarrow{\sim} \mathcal{D}^{\perf} (\Phi_\mathcal{S}).
\end{align*}
\end{remark}

Going back to the initial setup of this section, we have the following.
\begin{lemma}
Let $\phi:\Phi\rightarrow \Gamma$ be a homological epimorphism and  
\begin{align*}
    \mathcal{U}= {}^\perp(\phi_*(\mathcal{D}^b(\mmod\Gamma))).
\end{align*}
Then there is an equivalence of categories
\begin{align*}
    T: \mathcal{D}^b(\mmod\Phi)/\mathcal{U}\rightarrow \phi_*(\mathcal{D}^b(\mmod\Gamma)).
\end{align*}
\end{lemma}
\begin{proof}
Recall that $(\mathcal{U},\phi_*(\mathcal{D}^b(\mmod\Gamma)))$ is a stable $t$-structure by Lemma \ref{lemma_tstructure}. Then the result follows from \cite[Proposition 1.2 I(3)]{IKM}.
\end{proof}

Note that the above lemma resembles Theorem \ref{thmC}. As explained in Remark \ref{remarkD}, when $d=1$ in our setup, a morphism $\Phi\rightarrow\Gamma$ is a universal localization if and only if it is a homological epimorphism and we have an equivalence of categories
\begin{align*}
   T:\mathcal{T}^c\xrightarrow{\sim} \mathcal{D}^{\perf} (\Gamma).
\end{align*}
On the other hand, in our construction $(\mathcal{U},\mathcal{U}^{\perp})$ is a stable $t$-structure by Lemma \ref{lemma_tstructure} and $\mathcal{D}^b(\mmod\Phi)$ is idempotent complete, so $\mathcal{D}^b(\mmod\Phi)/\mathcal{U}$ is already idempotent complete. Moreover, because of our setup, we could equivalently write $\mathcal{D}^{\perf} (\Phi)$ or $\mathcal{D}^b (\mmod\Phi)$.

Taking inspiration from Definition \ref{defn_alternative}, in the next section we introduce a definition of universal localizations of a $d$-homological pair $(\Phi,\mathcal{F})$ for a more general choice of $\mathcal{U}$ that agrees with the above construction for the special case when
\begin{align*}
    \mathcal{U}= {}^\perp(\phi_*(\mathcal{D}^b(\mmod\Gamma)))
\end{align*}
for a homological epimorphism $\phi:\Phi\rightarrow\Gamma$ with an extra property.

%For higher values of $d$, in our setup we have gldim$\Phi\leq d$ and a $d$-homological pair $(\Phi,\mathcal{F})$. In the next section, we will define universal localizations of $(\Phi,\mathcal{F})$ in a way that they are in bijective correspondence with homological epimorphisms of $d$-homological pairs.

\section{Universal localizations of $d$-homological pairs}\label{section_univ_loc}
In this section, we introduce the definition of universal localizations of $d$-homological pairs. We show their connection to homological epimorphisms of $d$-homological pairs and prove their existence for any given $\mathcal{U}$ with the required properties.
%Before introducing this new theory, we quickly recall the classic theory of universal localizations.

%\begin{remark}\label{remark_B_extended}
%In Definition \ref{defn_univ_classic}, we recalled the classic definition of universal localization $\Phi\rightarrow\Phi_\mathcal{S}$ of $\Phi$ with respect to a set $\mathcal{S}$ of morphisms between finitely generated projective right $\Phi$-modules. Let $\mathcal{R}^c$ be the smallest triangulated subcategory of $\mathcal{D}^{b}(\Phi)$ containing $\mathcal{S}$ and all the direct summands of any of its objects.

%We mentioned in Remark \ref{remark_B} that the universal localization $\Phi\rightarrow\Phi_\mathcal{S}$ can alternatively be defined as the algebra homomorphism initial in the category of algebra homomomorphisms $\Phi\rightarrow\Lambda$ such that $\mathcal{R}^c\otimes^L_\Phi \Lambda=0$. In fact, for any morphism $\sigma\in\mathcal{S}$, we have that $\sigma\otimes^L_\Phi \Phi_\Lambda=0$ if and only if $\sigma\otimes_\Phi \Lambda$ is acyclic and hence invertible. By the way $\mathcal{R}^c$ is constructed, we conclude that $\mathcal{R}^c\otimes^L_\Phi \Lambda=0$ if and only if $\sigma\otimes_\Phi \Lambda$ is invertible for each $\sigma\in\mathcal{S}$. One can also easily see that the universal property from Definition \ref{defn_univ_classic} is equivalent to the initial property described above.
%Moreover, if $\Phi$ is hereditary, we saw in Remark \ref{remarkD} of the introduction that universal localizations of $\Phi$ are in bijection with homological epimorphisms starting at $\Phi$.
%\end{remark}

Working in Setup \ref{setup} and generalising the alternative definition of classic universal localizations presented in Definition \ref{defn_alternative}, we introduce the definition of universal localizations of $d$-homological pairs. Recall that since we work in categories of finitely generated modules and our algebras have finite global dimension, we can write $\mathcal{D}^b$ instead of $\mathcal{D}^{\perf}$.

\begin{defn}\label{defn_higher_univ_loc}
Let $\mathcal{U}\subseteq\mathcal{D}^b(\mmod \Phi)$ be a wide subcategory satisfying the following two conditions:
\begin{enumerate}
    \item\label{property1} $\mathcal{U}^\perp$ is functorially finite in $\mathcal{D}^b(\mmod \Phi)$,
    \item\label{property2} each object in $\mathcal{U}^\perp$ has its $\overline{\mathcal{F}}$-cover and its $\overline{\mathcal{F}}$-envelope in $\mathcal{U}^\perp\cap\overline{\mathcal{F}}$.
\end{enumerate}
A \textit{universal localization of $(\Phi, \mathcal{F})$ with respect to $\mathcal{U}$} is an algebra morphism $\Phi\xrightarrow{\phi}\Gamma$ initial in the category of algebra morphisms $\Phi\rightarrow \Lambda$ such that $\mathcal{U}\otimes^L_\Phi \Lambda=0$.
\end{defn}

\begin{remark}
Definition \ref{defn_higher_univ_loc} makes sense also if we drop the assumption gldim$\Phi\leq d$. Moreover,  if a universal localization of $(\Phi,\mathcal{F})$ with respect to some $\mathcal{U}$ with the given properties exists, then it is unique up to isomorphism. However, it is not clear whether it will always exist in such a general setup. 
We will show that if gldim$\Phi\leq d$, then such a construction exists for any given $\mathcal{U}$ with the required properties and we can hence refer to it as \textit{the} universal localization of $(\Phi, \mathcal{F})$ with respect to $\mathcal{U}$. 
\end{remark}

We show how in the base case the above definition coincides with the definition of universal localization of $\Phi$ with respect to a set of morphisms between finitely generated projectives $\mathcal{S}$.

\begin{proposition}\label{remark_basecase_coincides}
In the classic case, that is $d=1$ and $\Phi$ is a finite dimensional hereditary $k$-algebra, a universal localization in the sense of Definition \ref{defn_higher_univ_loc} is precisely a universal localization in the sense of Definition \ref{defn_alternative}. In particular, the objects of $\mathcal{U}$ are direct sums of shifts of $\Phi$-modules and $\mathcal{U}=\mathcal{R}^c_\mathcal{S}$, where $\mathcal{S}$ is the set of projective resolutions of these modules.
\end{proposition}
\begin{proof}
Since the only $1$-cluster tilting subcategory of $\mmod\Phi$ is $\mmod\Phi$ itself, the $1$-homological pair is forced to be $(\Phi,\Phi)$. Suppose that $\phi:(\Phi,\Phi)\rightarrow (\Gamma,\Gamma)$ is the universal localization of $(\Phi,\Phi)$ with respect to a wide subcategory $\mathcal{U}\subseteq\mathcal{D}^b(\mmod\Phi)$ with properties (\ref{property1}) and (\ref{property2}) in the sense of Definition \ref{defn_higher_univ_loc}. Since $\Phi$ is hereditary, any object in $\mathcal{D}^b(\mmod\Phi)$ is isomorphic to its cohomology.
%this is stated in Krause's book Homological theory of representations Prop 4.4.15 pp139
Hence the objects in $\mathcal{U}$ are direct sums of shifts of $\Phi$-modules, all having at most $2$-terms projective resolutions. Hence, choosing $\mathcal{S}$ to be the set of these projective resolutions, and using the same notation as in Remark \ref{remark_alternative_defn}, we have that $\mathcal{U}=\mathcal{R}^c_\mathcal{S}$.
Since the initial properties of the two definitions clearly coincide, $\phi:\Phi\rightarrow \Gamma$ is also the universal localization of $\Phi$ with respect to $\mathcal{S}$ in the sense of Definition \ref{defn_alternative}. 

Suppose now that $\phi:\Phi\rightarrow \Gamma$ is the universal localization of $\Phi$ with respect to some set $\mathcal{S}$ in the sense of Definition \ref{defn_alternative}. Then $\mathcal{U}:=\mathcal{R}^c_\mathcal{S}\subseteq\mathcal{D}^b(\mmod\Phi)$ is a wide subcategory by construction and the initial properties of the two definitions coincide. It only remains to check that properties (\ref{property1}) and (\ref{property2}) hold in this case.

We have that $\mathcal{F}=\mmod\Phi$ is the only choice of a $1$-cluster tilting subcategory, so that $\overline{\mathcal{F}}=\mathcal{D}^b(\mmod\Phi)$ and property (\ref{property2}) clearly holds.

By \cite[Theorem 6.1]{KS}, $\phi$ is a homological epimorphism and so
%Applying the classic theorem stated in \cite[Section 1.2]{HJV}, which combines the results \cite[Theorem 1.6.1(2)]{I3} and \cite[Theorem 4.8]{S}, we have that $\phi_*(\mmod\Gamma)\subseteq \mmod\Phi$ is a functorially finite wide subcategory. Then, by \cite[Theorem 1.1]{B},
%\begin{align*}
%\end{align*}
%is a wide subcategory. Using \cite[Theorem 5.1]{I}, it is easy to see that $\phi_*(\mathcal{D}^b(\mmod\Gamma))$ is functorially finite in $\mathcal{D}^b(\mmod\Phi)$.
\begin{align*}
({}^\perp(\phi_*(\mathcal{D}^b(\mmod\Gamma)),\phi_*(\mathcal{D}^b(\mmod\Gamma)))
\end{align*}
is a stable $t$-structure by Lemma \ref{lemma_tstructure}. Moreover, as seen in Remark \ref{remark_fullfaithful}, the functor
\begin{align*}
   \phi_*(-):\mathcal{D}^b(\mmod\Gamma)\rightarrow \mathcal{D}^b(\mmod\Phi) 
\end{align*}
has both a left and a right adjoint. Hence $\phi_*(\mathcal{D}^b(\mmod\Gamma))\subseteq \mathcal{D}^b(\mmod\Phi)$ is functorially finite by \cite[Lemma 4.3]{HJV} and its dual. Then, in order to prove that property (\ref{property1}) holds, it is enough to show that $(\mathcal{R}^c_\mathcal{S})^\perp=\phi_*(\mathcal{D}^b(\mmod\Gamma))$.
Given $X\in\mathcal{D}^b(\mmod\Phi)$ and $Y\in\mathcal{D}^b(\mmod\Gamma)$, since $\phi^*(-)=-\otimes^L_\Phi\Gamma$ is the left adjoint of $\phi_*(-)$, we have that
\begin{align*}
    \Hom_{\mathcal{D}^b(\mmod\Phi)}(X,\phi_*(Y))\cong \Hom _{\mathcal{D}^b(\mmod\Gamma)}(X\otimes^L_\Phi\Gamma,Y).
\end{align*}
Then, $X\in {}^\perp(\phi_*(\mathcal{D}^b(\mmod\Gamma))$ if and only if $X\otimes^L_\Phi \Gamma=0$.
Since $\mathcal{D}^b(\mmod\Phi)$ is idempotent complete, $\mathcal{D}^b(\mmod\Phi)/\mathcal{R}^c_\mathcal{S}$ is idempotent complete. Moreover, because of our setup, we can equivalently write $\mathcal{D}^{\perf} (\Phi)$ or $\mathcal{D}^b (\mmod\Phi)$. So, by Theorem \ref{thmC} we have canonical factorization of  $-\otimes^L_\Phi\Gamma: \mathcal{D}^b(\mmod\Phi)\rightarrow \mathcal{D}^b(\mmod\Gamma)$ as
\begin{align*}
    \mathcal{D}^b(\mmod\Phi)\xrightarrow{\pi} \mathcal{D}^b (\mmod\Phi)/\mathcal{R}^c_\mathcal{S}\xrightarrow{T}\mathcal{D}^b(\mmod\Gamma),
\end{align*}
where $\pi$ is the Verdier quotient map and $T$ is an equivalence of categories by Remark \ref{remarkD}. Hence $X\otimes^L_\Phi \Gamma=0$ if and only if $\pi(X)=0$ and so $X\in{}^\perp(\phi_*(\mathcal{D}^b(\mmod\Gamma))$ if and only if  $\pi(X)=0$.
Since $\pi(\mathcal{R}^c_\mathcal{S})=0$, by \cite[Proposition II.2.3.1(d)(d)bis]{VJ}, we conclude that $^\perp(\phi_*(\mathcal{D}^b(\mmod\Gamma))=\mathcal{R}^c_\mathcal{S}$. 
 Hence $(\mathcal{R}^c_\mathcal{S})^\perp= \phi_*(\mathcal{D}^b(\mmod\Gamma))$ and property (\ref{property1}) holds.
%Hence, in the case $d=1$, the universal localization of $(\Phi,\Phi)$ with respect to a wide subcategory $\mathcal{U}$ coincides with Definition \ref{defn_alternative} of the universal localization of $\Phi$ with respect to the corresponding set $\mathcal{S}$.
\end{proof}

%Now that we have a definition for universal localizations of $d$-homological pairs, we can focus on proving that \ref{itemd} from Theorem \ref{thm} is in bijection with the others sets of the theorem.
\subsection{From homological epimorphisms to universal localizations}
We start with a homological epimorphism of $d$-homological pairs $(\Phi, \mathcal{F})\xrightarrow{\phi}(\Gamma, \mathcal{G})$ and show $\phi$ is the universal localization of $(\Phi, \mathcal{F})$ with respect to some $\mathcal{U}$.

\begin{lemma}\label{lemma_univ_intersection}
Let  $(\Phi, \mathcal{F})\xrightarrow{\phi}(\Gamma, \mathcal{G})$ be a homological epimorphism of $d$-homological pairs and
\begin{align*}
    \mathcal{U}= {}^\perp(\phi_*(\mathcal{D}^b(\mmod\Gamma))).
\end{align*}
Then $\mathcal{U}^\perp\cap\overline{\mathcal{F}}=\phi_*(\overline{\mathcal{G}})$ and it is a functorially finite wide subcategory of $\overline{\mathcal{F}}$.
\end{lemma}

\begin{proof}
Since $\phi_*:\mathcal{D}^b(\mmod\Gamma)\rightarrow\mathcal{D}^b(\mmod\Phi)$ is triangulated by Remark \ref{remark_fullfaithful} and $\phi_*(\overline{\mathcal{G}})$ is closed under direct summands, we have that
\begin{align*}
    \phi_*(\overline{\mathcal{G}})=\phi_*(\text{add }(\Sigma^{d\mathbb{Z}}\mathcal{G}))=\text{add } (\Sigma^{d\mathbb{Z}}\phi_*(\mathcal{G}))=\overline{\phi_*(\mathcal{G})}.
\end{align*}
Moreover, $\phi_*(\mathcal{G})\subseteq\mathcal{F}$ is functorially finite and wide by \cite[Proposition 6.2]{HJV}. Then, by Proposition \ref{lemma_ab} we have that $\phi_*(\overline{\mathcal{G}})=\overline{\phi_*(\mathcal{G})}\subseteq \overline{\mathcal{F}}$ is functorially finite and wide.

It remains to show that $\mathcal{U}^\perp\cap\overline{\mathcal{F}}=\phi_*(\overline{\mathcal{G}})$. Recall that $\overline{\mathcal{G}}\subseteq \mathcal{D}^b(\mmod\Gamma)$, so that
\begin{align*}
    \phi_*(\overline{\mathcal{G}})\subseteq\phi_*(\mathcal{D}^b(\mmod\Gamma))=\mathcal{U}^\perp
\end{align*}
and the inclusion $\phi_*(\overline{\mathcal{G}})\subseteq \mathcal{U}^\perp\cap\overline{\mathcal{F}}$ is clear.

Consider now $\overline{f}\in\mathcal{U}^\perp\cap\overline{\mathcal{F}}$. By \cite[Lemma 4.3]{HJV}, we have that 
\begin{align*}
    \overline{f}\rightarrow \phi_*\phi^*(\overline{f})
\end{align*}
is a strong $\mathcal{U}^\perp$-envelope. But since $\overline{f}\in\mathcal{U}^\perp$, we have that $\overline{f}\cong \phi_*\phi^*(\overline{f})$ and it is enough to show that $\phi_*\phi^*(\overline{f})\in\phi_*(\overline{\mathcal{G}})$.

Recall that by Remark \ref{remark_fullfaithful}, $\phi_*:\mathcal{D}^b(\mmod\Gamma)\rightarrow\mathcal{D}^b(\mmod\Phi)$ is triangulated and it has left adjoint functor $\phi^*:\mathcal{D}^b(\mmod\Phi)\rightarrow\mathcal{D}^b(\mmod\Gamma)$. Then, by \cite[Lemma 5.3.6]{N}, $\phi^*$ is also a triangulated functor and
\begin{align*}
    \phi^*(\overline{\mathcal{F}})=\phi^*(\text{add }(\Sigma^{d\mathbb{Z}}\mathcal{F}))=\text{add } (\Sigma^{d\mathbb{Z}}\phi^*(\mathcal{F}))=\overline{\phi^*(\mathcal{F})}.
\end{align*}
Since $\phi^*(\mathcal{F})\subseteq\mathcal{G}$ by \cite[Proposition 5.5]{HJV}, we have that $\phi^*(\overline{\mathcal{F}})\subseteq\overline{\mathcal{G}}$. In particular, $\phi_*\phi^*(\overline{f})\in\phi_*(\overline{\mathcal{G}})$ as we wished to prove.
\end{proof}

\begin{lemma}\label{lemma_funct_fin}
Let  $(\Phi, \mathcal{F})\xrightarrow{\phi}(\Gamma, \mathcal{G})$ be a homological epimorphism of $d$-homological pairs and
\begin{align*}
    \mathcal{U}= {}^\perp(\phi_*(\mathcal{D}^b(\mmod\Gamma))).
\end{align*}
Then $\mathcal{U}^\perp$ is functorially finite in $\mathcal{D}^b(\mmod\Phi)$.
\end{lemma}
\begin{proof}
By construction, we have that $\mathcal{U}^\perp=\phi_*(\mathcal{D}^b(\mmod\Gamma))$. As seen in Remark \ref{remark_fullfaithful}, $\phi_*:\mathcal{D}^b(\mmod\Gamma)\rightarrow\mathcal{D}^b(\mmod\Phi)$ has left and right adjoint functors. Then, by \cite[Lemma 4.3]{HJV} and its dual, we have that $\mathcal{U}^\perp$ is both strongly enveloping and strongly covering in $\mathcal{D}^b(\mmod\Phi)$. In particular, $\mathcal{U}^\perp$ is functorially finite in $\mathcal{D}^b(\mmod\Phi)$.
\end{proof}

\begin{lemma}\label{lemma_prop2}
Let  $(\Phi, \mathcal{F})\xrightarrow{\phi}(\Gamma, \mathcal{G})$ be a homological epimorphism of $d$-homological pairs and
\begin{align*}
    \mathcal{U}= {}^\perp(\phi_*(\mathcal{D}^b(\mmod\Gamma))).
\end{align*}
Then each object in $\mathcal{U}^\perp$ has its $\overline{\mathcal{F}}$-cover and its $\overline{\mathcal{F}}$-envelope in $\mathcal{U}^\perp\cap\overline{\mathcal{F}}$.
\end{lemma}

\begin{proof}
We only prove that each object in $\mathcal{U}^\perp$ has its $\overline{\mathcal{F}}$-cover in $\mathcal{U}^\perp\cap\overline{\mathcal{F}}$. The corresponding statement on $\overline{\mathcal{F}}$-envelopes follows by a dual argument.

Consider an arbitrary element $x$ in $\mathcal{U}^\perp=\phi_*(\mathcal{D}^b(\mmod\Gamma))$, that is $x=\phi_*(y)$ for some $y\in \mathcal{D}^b(\mmod\Gamma)$.
Since $\overline{\mathcal{G}}\subseteq\mathcal{D}^b(\mmod\Gamma)$ is $d$-cluster tilting, we can build a tower of triangles in $\mathcal{D}^b(\mmod\Gamma)$ of the form
\begin{align*}
\begin{gathered}
\xymatrix@!0 @C=3em @R=3em{
& \overline{g}_{d-2}\ar[rd]\ar[rr] && \overline{g}_{d-3}\ar[rd]\ar[r]&&\cdots&\ar[r]& \overline{g}_{1}\ar[rr]\ar[rd]&&\overline{g}_{0} \ar[rd]
\\
\overline{g}_{d-1}\ar[ru]&& c_{d-2}\ar@{~>}[ll]\ar[ru]&& c_{d-3}\ar@{~>}[ll]&\cdots& c_2\ar[ru]&& c_1 \ar@{~>}[ll]\ar[ru]&& y,\ar@{~>}[ll]
}
\end{gathered}
\end{align*}
where $\overline{g}_i\in\overline{\mathcal{G}}$ and, letting $c_0:=y$, $\overline{g}_i\rightarrow c_i$ is a $\overline{\mathcal{G}}$-cover for each $0\leq i\leq d-2$. Note that
\begin{align*}
    c_1\in\overline{\mathcal{G}}*\Sigma \overline{\mathcal{G}}*\cdots *\Sigma^{d-2}\overline{\mathcal{G}}.
\end{align*}
Applying the triangulated functor $\phi_*$ to the rightmost triangle in the tower, we obtain the triangle in $\mathcal{D}^b(\mmod\Phi)$
\begin{align*}
    \phi_*(c_1)\rightarrow \phi_*(\overline{g}_0)\xrightarrow{\gamma} x\xrightarrow{\xi} \Sigma\phi_*(c_1),
\end{align*}
where it is easy to see that $\gamma$ is a $\phi_*(\overline{\mathcal{G}})$-cover. Since $\phi_*(\overline{\mathcal{G}})\subseteq\overline{\mathcal{F}}$, we have that
\begin{align*}
    \Sigma \phi_*(c_1)\in\Sigma\overline{\mathcal{G}}*\Sigma^2 \overline{\mathcal{G}}*\cdots *\Sigma^{d-1}\overline{\mathcal{G}}\subseteq \Sigma\overline{\mathcal{F}}*\Sigma^2 \overline{\mathcal{F}}*\cdots *\Sigma^{d-1}\overline{\mathcal{F}}.
\end{align*}
Hence, since $\overline{\mathcal{F}}\subseteq \mathcal{D}^b(\mmod\Phi)$ is $d$-cluster tilting, we have that
\begin{align*}
    \Hom_{\mathcal{D}^b(\mmod\Phi)}(\overline{\mathcal{F}}, \Sigma \phi_*(c_1))=0.
\end{align*}
Then, given any morphism $\alpha:\overline{f}\rightarrow x$ with $\overline{f}\in\overline{\mathcal{F}}$, we have that $\xi\circ\alpha=0$ and so $\alpha$ factors through $\gamma$. This proves that $\gamma$ is an $\overline{\mathcal{F}}$-cover of $x$. Since $x$ was chosen arbitrarily in $\mathcal{U}^\perp$, we conclude that every object in $\mathcal{U}^\perp$ has its $\overline{\mathcal{F}}$-cover in $\phi_*(\overline{\mathcal{G}})$. This completes the proof because $\mathcal{U}^\perp\cap\overline{\mathcal{F}}=\phi_*(\overline{\mathcal{G}})$ by Lemma \ref{lemma_univ_intersection}.
\end{proof}

We prove there is an injection of the set \ref{itemc} into the set \ref{itemd} of Theorem \ref{thm}.
\begin{proposition}\label{proposition_homoepi_to_univloc}
There is an injection
\begin{align*}
   \Bigg\{
    \begin{matrix} \text{equivalence classes of homological}\\ \text{ epimorphisms of $d$-homological pairs} \\ (\Phi, \mathcal{F})\rightarrow(\Gamma, \mathcal{G}) \end{matrix}
    \Bigg\} 
    \xrightarrow{\Psi}
    \Bigg\{
    \begin{matrix} \text{equivalence classes of}\\ \text{universal localizations} \\ \text{of } (\Phi,\mathcal{F}) \end{matrix}
    \Bigg\}
    \end{align*}
sending a homological epimorphism of $d$-homological pairs $(\Phi, \mathcal{F})\xrightarrow{\phi}(\Gamma, \mathcal{G})$ to itself. In fact, $\Phi\xrightarrow{\phi}\Gamma$ is the universal localization of $(\Phi,\mathcal{F})$ with respect to ${}^\perp(\phi_*(\mathcal{D}^b(\mmod\Gamma)))$. 
\end{proposition}

\begin{proof}
Let  $(\Phi, \mathcal{F})\xrightarrow{\phi}(\Gamma, \mathcal{G})$ be a homological epimorphism of $d$-homological pairs and
\begin{align*}
    \mathcal{U}= {}^\perp(\phi_*(\mathcal{D}^b(\mmod\Gamma))).
\end{align*}
We show that $\phi: \Phi\rightarrow \Gamma$ is the universal localization of $(\Phi, \mathcal{F})$ with respect to $\mathcal{U}$. Note that this is enough, because if the map between sets is well-defined, then it is clearly an injection.

We have that $\mathcal{U}\subseteq\mathcal{D}^b(\mmod\Phi)$ is a wide subcategory by Lemma \ref{lemma_U_wide}. Property (\ref{property1}) from Definition \ref{defn_higher_univ_loc} is satisfied by Lemma \ref{lemma_funct_fin} and property (\ref{property2}) by Lemma \ref{lemma_prop2}. Moreover, by Lemma \ref{lemma_initial_prop}, $\phi$ is initial in the category of algebra morphisms $\Phi\rightarrow\Lambda$ such that $\mathcal{U}\otimes^L_\Phi\Lambda=0$. Hence, $\phi$ is the universal localization of $(\Phi,\mathcal{F})$ with respect to $\mathcal{U}$.
\end{proof}

\subsection{From universal localizations to homological epimorphisms}
 For the rest of the section, assume that $\mathcal{U}\subseteq\mathcal{D}^b(\mmod \Phi)$ is a wide subcategory satisfying properties (\ref{property1}) and (\ref{property2}) from Definition \ref{defn_higher_univ_loc}. We prove that the map from Proposition \ref{proposition_homoepi_to_univloc} is surjective and that the universal localization of $(\Phi,\mathcal{F})$ with respect to $\mathcal{U}$ exists.

\begin{lemma}\label{lemma_U_triangulated}
The subcategory $\mathcal{U}^\perp\subseteq \mathcal{D}^b(\mmod\Phi)$ is triangulated.
\end{lemma}

\begin{proof}
Given $v\in\mathcal{U}^\perp$, $u\in\mathcal{U}$ and $i\in\mathbb{Z}$, we have that 
\begin{align*}
    \Hom_{\mathcal{D}^b(\mmod\Phi)}(u, \Sigma^i v)\cong 
    \Hom_{\mathcal{D}^b(\mmod\Phi)}(\Sigma^{-i} u, v)=0,
\end{align*}
where the last equality holds because $\mathcal{U}$ is closed under $\Sigma^{\pm 1}$. Hence $\Sigma^i v\in \mathcal{U}^\perp$ and $\mathcal{U}^\perp$ is closed under $\Sigma^{\pm 1}$.

In order to prove the 2 out of 3 property, it is then enough to prove closure under extensions. Let
\begin{align*}
    y\rightarrow t\rightarrow x\rightarrow \Sigma y
\end{align*}
be a triangle in $\mathcal{D}^b(\mmod\Phi)$ with $y$ and $x$ in $\mathcal{U}^\perp$. Given $u\in\mathcal{U}$, we have an exact sequence of the form 
\begin{align*}
   \cdots\rightarrow \Hom_{\mathcal{D}^b(\mmod\Phi)}(u,y)
   \rightarrow \Hom_{\mathcal{D}^b(\mmod\Phi)}(u,t)
   \rightarrow \Hom_{\mathcal{D}^b(\mmod\Phi)}(u,x)\rightarrow\cdots.
\end{align*}
Note that the first and the last term of the above sequence are zero, then so is the second one and we conclude that $t\in\mathcal{U}^\perp$. Hence $\mathcal{U}^\perp\subseteq \mathcal{D}^b(\mmod\Phi)$ is closed under extensions and triangulated.
\end{proof}

\begin{lemma}\label{lemma_UUperp_stablet}
The pair $(\mathcal{U}, \mathcal{U}^\perp)$ is a stable $t$-structure in the sense of \cite[definition 9.14]{M}.
\end{lemma}

\begin{proof}
We have that $\mathcal{U}$ is a wide subcategory by assumption and by Lemma \ref{lemma_U_triangulated}, $\mathcal{U}^\perp$ is a triangulated subcategory. It remains to show that $\mathcal{U}*\mathcal{U}^\perp =\mathcal{D}^b(\mmod\Phi)$. Let $x\in\mathcal{D}^b(\mmod\Phi)$. Since $\mathcal{U}^\perp$ has property (\ref{property1}), there exists a $\mathcal{U}^\perp$-envelope of $x$, say $\xi:x\rightarrow u$. Complete $\xi$ to a triangle:
\begin{align*}
    y\rightarrow x\xrightarrow{\xi} u\rightarrow \Sigma y.
\end{align*}
Since $\mathcal{U}^\perp\subseteq\mathcal{D}^b(\mmod\Phi)$ is full, closed under extensions and direct summands, by the dual of \cite[Lemma 2.1]{PJ} we have that
\begin{align*}
    \Hom_{\mathcal{D}^b(\mmod\Phi)}(y, \mathcal{U}^\perp)=0.
\end{align*}
Hence $y\in\mathcal{U}$ and $\mathcal{U}*\mathcal{U}^\perp =\mathcal{D}^b(\mmod\Phi)$. 
\end{proof}

\begin{lemma}\label{lemma_ffinite}
The subcategory $\overline{\mathcal{F}}\cap\mathcal{U}^\perp$ is functorially finite in $\overline{\mathcal{F}}$.
\end{lemma}
\begin{proof}
We only prove that $\overline{\mathcal{F}}\cap\mathcal{U}^\perp$ is precovering in $\overline{\mathcal{F}}$, the proof that it is preenveloping is dual.
Let $\overline{f}\in\overline{\mathcal{F}}\subseteq \mathcal{D}^b(\mmod\Phi)$. By property (\ref{property1}), there exists a $\mathcal{U}^\perp$-precover of $\overline{f}$, say $\nu: v\rightarrow \overline{f}$. Since $\overline{\mathcal{F}}\subseteq \mathcal{D}^b(\mmod\Phi)$ is functorially finite, there is an $\overline{\mathcal{F}}$-precover of $v$, say $\epsilon: \overline{e}\rightarrow v$. By property (\ref{property2}), $\overline{e}\in\overline{\mathcal{F}}\cap\mathcal{U}^\perp$. Let $z\in\overline{\mathcal{F}}\cap\mathcal{U}^\perp$ and $\zeta: z\rightarrow \overline{f}$. Then, since $\nu$ is a $\mathcal{U}^\perp$-precover, there exists a morphism $\alpha: z\rightarrow v$ such that $\nu\circ\alpha=\zeta$. Since $\epsilon$ is an $\overline{\mathcal{F}}$-precover, then there is a morphism $\beta:z\rightarrow \overline{e}$ such that $\epsilon\circ\beta=\alpha$. Hence
\begin{align*}
    \nu\circ\epsilon\circ\beta=\nu\circ\alpha=\zeta
\end{align*}
and we conclude that $\nu\circ\epsilon$ is a $(\overline{\mathcal{F}}\cap\mathcal{U}^\perp)$-precover and $\overline{\mathcal{F}}\cap\mathcal{U}^\perp$ is precovering.
\end{proof}

\begin{lemma}\label{lemma_wide}
The subcategory $\overline{\mathcal{F}}\cap\mathcal{U}^\perp$ is wide in $\overline{\mathcal{F}}$.
\end{lemma}

\begin{proof}
Since both $\mathcal{U}^\perp$ and $\overline{\mathcal{F}}$ are closed under direct sums and direct summands, then so is $\overline{\mathcal{F}}\cap\mathcal{U}^\perp$. Moreover, $\mathcal{U}^{\perp}$ is closed under $\Sigma^{\pm 1}$ by Lemma \ref{lemma_UUperp_stablet} and $\overline{\mathcal{F}}$ is closed under $\Sigma^{\pm d}$ because it is $(d+2)$-angulated. Hence $\overline{\mathcal{F}}\cap\mathcal{U}^\perp$ is closed under $\Sigma^{\pm d}$.

It remains to show that $\overline{\mathcal{F}}\cap\mathcal{U}^\perp$ is closed under $d$-extensions.
Let $\alpha: a\rightarrow\Sigma^d b$ be a morphism in $\overline{\mathcal{F}}\cap\mathcal{U}^\perp$. Since $\overline{\mathcal{F}}\subseteq\mathcal{D}^b(\mmod\Phi)$ is $d$-cluster tilting, we can build a tower of triangles in $\mathcal{D}^b(\mmod\Phi)$ of the form

\begin{align*}
\begin{gathered}
\xymatrix@!0 @C=3em @R=3em{
& \overline{f}^{1}\ar[rd]\ar[rr] && \overline{f}^{2}\ar[rd]\ar[r]&&\cdots&\ar[r]& \overline{f}^{d-1}\ar[rr]\ar[rd]&& a \ar[rd]
\\
\overline{f}^{0}\ar[ru]&& y^{1}\ar@{~>}[ll]\ar[ru]&& y^{2}\ar@{~>}[ll]&\cdots& y^{d-2}\ar[ru]&& y^{d-1} \ar@{~>}[ll]\ar[ru]&& \Sigma^db,\ar@{~>}[ll]
}
\end{gathered}
\end{align*}
where $\overline{f}^i\in\overline{\mathcal{F}}$ and $\overline{f}^i\rightarrow y^i$ is an $\overline{\mathcal{F}}$-cover for each $1\leq i\leq d-1$. 
This gives a $(d+2)$-angle in $\overline{\mathcal{F}}$
\begin{align*}
    \epsilon:\,\,\overline{f}^0\rightarrow \overline{f}^1\rightarrow\cdots\rightarrow \overline{f}^{d-1}\rightarrow a\xrightarrow{\alpha} \Sigma^d b\rightarrow\Sigma^d \overline{f}^0.
\end{align*}
Since $a$ and $\Sigma^d b$ are in $\mathcal{U}^\perp$ and $\mathcal{U}^\perp$ is a triangulated subcategory of $\mathcal{D}^b(\mmod\Phi)$ by Lemma \ref{lemma_U_triangulated}, then $y^{d-1}\in\mathcal{U}^\perp$. By property (\ref{property2}), we then have that $\overline{f}^{d-1}\in\overline{\mathcal{F}}\cap\mathcal{U}^\perp$. Repeating the above argument, we deduce that also
\begin{align*}
    \overline{f}^{d-2}, \overline{f}^{d-3}, \dots, \overline{f}^{1}, \overline{f}^{0}\in \overline{\mathcal{F}}\cap\mathcal{U}^\perp.
\end{align*}
Hence $\epsilon$ is a $(d+2)$-angle in $\overline{\mathcal{F}}$ with all of its objects in $\overline{\mathcal{F}}\cap\mathcal{U}^\perp$ and so is its rotation:
\begin{align*}
   b\rightarrow\overline{f}^0\rightarrow \overline{f}^1\rightarrow\cdots\rightarrow \overline{f}^{d-1}\rightarrow a\xrightarrow{\alpha} \Sigma^d b
\end{align*}
and we conclude that $\overline{\mathcal{F}}\cap\mathcal{U}^\perp$ is closed under $d$-extensions.
\end{proof}

\begin{proposition}\label{prop_map_surjective}
The map $\Psi$ from Proposition \ref{proposition_homoepi_to_univloc} is a surjection.
\end{proposition}

\begin{proof}
Let $\mathcal{U}\subseteq\mathcal{D}^b(\mmod \Phi)$ be a wide subcategory satisfying properties (\ref{property1}) and (\ref{property2}) from Definition \ref{defn_higher_univ_loc}.
We prove that the universal localization of $(\Phi,\mathcal{F})$ with respect to $\mathcal{U}$ exists.
By Lemmas \ref{lemma_ffinite} and \ref{lemma_wide}, we have that $\mathcal{U}^\perp\cap\overline{\mathcal{F}}$ is functorially finite and wide in $\overline{\mathcal{F}}$. Then, by Proposition \ref{lemma_ab}, $\mathcal{U}^\perp\cap\mathcal{F}$ is functorially finite and wide in $\mathcal{F}$ and by Proposition \ref{lemma_ac}, there is a homological epimorphism of $d$-homological pairs
\begin{align*}
    (\Phi,\mathcal{F})\xrightarrow{\phi}(\Gamma, \mathcal{G}).
\end{align*}
 Moreover, by Proposition \ref{proposition_homoepi_to_univloc}, we have that $\phi$ is the universal localization of $(\Phi,\mathcal{F})$ with respect to ${}^\perp (\phi_*(\mathcal{D}^b(\mmod\Gamma)))$. It remains to show that $\mathcal{U}={}^\perp (\phi_*(\mathcal{D}^b(\mmod\Gamma)))$, or equivalently that $\mathcal{U}^\perp= \phi_*(\mathcal{D}^b(\mmod\Gamma))$.

The homological epimorphism $\phi$ is built as explained in Remark \ref{remark_construction6.3} with $\mathcal{W}=\mathcal{U}^\perp\cap\mathcal{F}$, and it has the property that $\phi_*(\mathcal{G})=\mathcal{U}^\perp\cap\mathcal{F}$.
Since $\phi_*:\mathcal{D}^b(\mmod\Gamma)\rightarrow\mathcal{D}^b(\mmod\Phi)$ is triangulated, $\phi_*(\overline{\mathcal{G}})$ is closed under direct summands and $\mathcal{U}^\perp$ is triangulated, we have that
\begin{align*}
    \phi_*(\overline{\mathcal{G}})=
    \phi_*(\text{add}(\Sigma^{d\mathbb{Z}}\mathcal{G}))=\text{add} (\Sigma^{d\mathbb{Z}}\phi_*(\mathcal{G}))=
    \text{add} (\Sigma^{d\mathbb{Z}}(\mathcal{U}^\perp\cap\mathcal{F}))=\mathcal{U}^\perp\cap\overline{\mathcal{F}}.
\end{align*}

We first prove that $\phi_*(\mathcal{D}^b(\mmod\Gamma))\subseteq \mathcal{U}^\perp$. Consider an arbitrary element $x$ in $\phi_*(\mathcal{D}^b(\mmod\Gamma))$, that is $x=\phi_*(y)$ for some $y\in \mathcal{D}^b(\mmod\Gamma)$.
Since $\overline{\mathcal{G}}\subseteq\mathcal{D}^b(\mmod\Gamma)$ is $d$-cluster tilting, we can build a tower of triangles in $\mathcal{D}^b(\mmod\Gamma)$ of the form
\begin{align*}
\begin{gathered}
\xymatrix@!0 @C=3em @R=3em{
& \overline{g}_{d-2}\ar[rd]\ar[rr] && \overline{g}_{d-3}\ar[rd]\ar[r]&&\cdots&\ar[r]& \overline{g}_{1}\ar[rr]\ar[rd]&&\overline{g}_{0} \ar[rd]
\\
\overline{g}_{d-1}\ar[ru]&& c_{d-2}\ar@{~>}[ll]\ar[ru]&& c_{d-3}\ar@{~>}[ll]&\cdots& c_2\ar[ru]&& c_1 \ar@{~>}[ll]\ar[ru]&& y,\ar@{~>}[ll]
}
\end{gathered}
\end{align*}
where $\overline{g}_i\in\overline{\mathcal{G}}$. Applying $\phi_*$ to this tower, we obtain the tower of triangles in $\mathcal{D}^b(\mmod\Phi)$
\begin{align*}
\begin{gathered}
\xymatrix@!0 @C=3em @R=3em{
& \phi_*(\overline{g}_{d-2})\ar[rd]\ar[rr] && \phi_*(\overline{g}_{d-3})\ar[rd]\ar[r]&&\cdots&\ar[r]& \phi_*(\overline{g}_{1})\ar[rr]\ar[rd]&&\phi_*(\overline{g}_{0}) \ar[rd]
\\
\phi_*(\overline{g}_{d-1})\ar[ru]&& \phi_*(c_{d-2})\ar@{~>}[ll]\ar[ru]&& \phi_*(c_{d-3})\ar@{~>}[ll]&\cdots& \phi_*(c_2)\ar[ru]&& \phi_*(c_1) \ar@{~>}[ll]\ar[ru]&& x.\ar@{~>}[ll]
}
\end{gathered}
\end{align*}
In particular, the leftmost triangle has the objects $\phi_*(c_{d-2})$ and $\phi_*(\overline{g}_{d-1})$ in
$\phi_*(\overline{\mathcal{G}})=\mathcal{U}^\perp\cap\overline{\mathcal{F}}$. Since $\mathcal{U}^\perp$ is triangulated, this implies that $\phi_*(c_{d-2})\in\mathcal{U}^\perp$. By induction on the triangles in the tower, we conclude that $x\in\mathcal{U}^\perp$ and so $\phi_*(\mathcal{D}^b(\mmod\Gamma))\subseteq \mathcal{U}^\perp$.

For the second inclusion, let $a\in\mathcal{U}^\perp$. Since $\overline{\mathcal{F}}\subseteq\mathcal{D}^b(\mmod\Phi)$ is $d$-cluster tilting, we can build a tower of triangles in $\mathcal{D}^b(\mmod\Phi)$ of the form
\begin{align*}
\begin{gathered}
\xymatrix@!0 @C=3em @R=3em{
& \overline{f}_{d-2}\ar[rd]\ar[rr] && \overline{f}_{d-3}\ar[rd]\ar[r]&&\cdots&\ar[r]& \overline{f}_{1}\ar[rr]\ar[rd]&&\overline{f}_{0} \ar[rd]
\\
\overline{f}_{d-1}\ar[ru]&& b_{d-2}\ar@{~>}[ll]\ar[ru]&& b_{d-3}\ar@{~>}[ll]&\cdots& b_2\ar[ru]&& b_1 \ar@{~>}[ll]\ar[ru]&& a,\ar@{~>}[ll]
}
\end{gathered}
\end{align*}
where $\overline{f}_i\in\overline{\mathcal{F}}$ and, letting $b_0:=a$, $\overline{f}_i\rightarrow b_i$ is an $\overline{\mathcal{F}}$-cover for each $0\leq i\leq d-2$. Recall that $\mathcal{U}$ has property (\ref{property2}) from Definition \ref{defn_higher_univ_loc} by assumption. Hence $\overline{f_0}\in\overline{\mathcal{F}}\cap \mathcal{U}^\perp$ and since $\mathcal{U}^\perp$ is triangulated, we have that $b_1\in\mathcal{U}^{\perp}$. Proceeding by induction, we conclude that $\overline{f}_i\in\overline{\mathcal{F}}\cap \mathcal{U}^\perp$ for each $0\leq i\leq d-2$. Recall that 
\begin{align*}
    \mathcal{U}^\perp\cap\overline{\mathcal{F}}=\phi_*(\overline{\mathcal{G}})\subseteq \phi_*(\mathcal{D}^b(\mmod\Gamma))
\end{align*}
and $\phi_*(\mathcal{D}^b(\mmod\Gamma))$ is a triangulated subcategory of $\mathcal{D}^b(\mmod\Phi)$. Then $\overline{f}_{d-1},\, \overline{f}_{d-2}\in\phi_*(\overline{\mathcal{G}})$ imply that $b_{d-2}\in \phi_*(\mathcal{D}^b(\mmod\Gamma))$. By induction on the triangles in the tower, we conclude that $a\in\phi_*(\mathcal{D}^b(\mmod\Gamma))$ and so $\mathcal{U}^\perp\subseteq \phi_*(\mathcal{D}^b(\mmod\Gamma))$.

Hence, specifying a wide subcategory $\mathcal{U}\subseteq\mathcal{D}^b(\mmod\Phi)$ with the required properties from Definition \ref{defn_higher_univ_loc} is enough to describe the universal localization $\phi: \Phi\rightarrow \Gamma$ of $(\Phi, \mathcal{F})$ with respect to $\mathcal{U}$ and $\phi$ coincides with the homological epimorphism of $d$-homological pairs
\begin{align*}
    (\Phi, \mathcal{F})\xrightarrow{\phi}(\Gamma, \mathcal{G})
\end{align*}
with $\mathcal{G}=\phi^!(\mathcal{U}^{\perp}\cap\mathcal{F})$. Hence the map $\Psi$ is also surjective
\end{proof}

\begin{remark}\label{prop_univ_loc_exists}
%Let $\mathcal{U}\subseteq \mathcal{D}^b(\mmod\Phi)$ be a wide subcategory satisfying properties (\ref{property1}) and (\ref{property2}) from Definition \ref{defn_higher_univ_loc}.
Note that the proof of Proposition \ref{prop_map_surjective} shows that the universal localization of $(\Phi, \mathcal{F})$ with respect to $\mathcal{U}$ exists and it coincides with the homological epimorphism of $d$-homological pairs described in Remark \ref{remark_construction6.3}  with $\mathcal{W}=\mathcal{U}^\perp\cap\mathcal{F}$.
\end{remark}

\section{The bijective correspondence}\label{section_proof}
Putting together the results proved throughout the paper, we prove our main theorem.
\begin{theorem}\label{thm}
In the situation of Setup \ref{setup}, there are bijections between  the following sets.
\begin{enumerate}[label=(\alph*)]
    \item\label{itema} Functorially finite wide subcategories of $\mathcal{F}$.
    \item\label{itemb} Functorially finite wide subcategories of $\overline{\mathcal{F}}$.
    \item\label{itemc} Equivalence classes of homological epimorphisms of $d$-homological pairs $(\Phi, \mathcal{F})\xrightarrow{\phi}(\Gamma, \mathcal{G})$.
    \item\label{itemd} Equivalence classes of universal localizations $\Phi\xrightarrow{\phi}\Gamma$ of $(\Phi, \mathcal{F})$.
\end{enumerate}
\end{theorem}
\begin{proof}
The sets \ref{itema} and \ref{itemb} are in bijection by Proposition \ref{lemma_ab} and the sets \ref{itema} and \ref{itemc} are in bijection by Proposition \ref{lemma_ac}.
Finally, the sets \ref{itemc} and \ref{itemd} are in bijection by Propositions \ref{proposition_homoepi_to_univloc} and \ref{prop_map_surjective}.
\end{proof}

\section{Example}\label{section_example}
We illustrate Theorem \ref{thm} in an example coming from the class of examples introduced by Vaso, describing all the sets in bijection, see \cite[Section 4]{V} and \cite[Section 7]{HJV}. We fix $m=3$ and $d=l=2$ in Vaso's example, so that 
\begin{align*}
    \Phi=kA_3/(\rad_{kA_3})^2.
\end{align*}
In other words, $\Phi$ is the path algebra of the quiver with relation
\begin{align*}
\xymatrix{
3\ar[r]\ar@{--}@/_1pc/[rr]& 2\ar[r]& 1.
}
\end{align*}

The Auslander-Reiten quiver of $\mmod\Phi$ is 
\begin{align*}
    \xymatrix@C=1em @R=1em{
    & f_2\ar[rd] &&f_3\ar[rd] &\\
    f_1\ar[ru]&& s_2\ar[ru] && f_4,
    }
\end{align*}
where $f_1$ is projective, $f_2$, $f_3$ are projective-injective, $f_4$ is injective and $s_2$ is  simple. Moreover,
\begin{align*}
    \mathcal{F}=\text{add}\{ f_1,f_2,f_3,f_4 \}\subseteq \mmod\Phi
\end{align*}
is a $2$-cluster tilting subcategory and $\Phi$ has global dimension $2$, so Setup \ref{setup} is satisfied.

In Sections \ref{section_example_a}, \ref{section_example_b}, \ref{section_example_c}, \ref{section_example_d} we describe, respectively, the elements of the sets (a), (b), (c) and (d) from Theorem \ref{thm} in this example.

\subsection{Functorially finite wide subcategories of $\mathcal{F}$}\label{section_example_a}
These have been described by Herschend, J{\o}rgensen and Vaso in \cite[Section 7]{HJV}, so the elements of set \ref{itema} are:
\begin{align*}
    &\mathcal{W}_1=\text{add}\{ f_1\},\\ &\mathcal{W}_2=\text{add}\{ f_2\}, \\
    &\mathcal{W}_3=\text{add}\{ f_3\}, \\ &\mathcal{W}_4=\text{add}\{ f_4\} \\
    &\mathcal{W}_5=\text{add}\{ f_1, f_3\},\\ &\mathcal{W}_6=\text{add}\{ f_2, f_4\}, \\
    &\mathcal{W}_7=\mathcal{F}.
\end{align*}
Note that these are all semisimple subcategories apart from $\mathcal{W}_7$.

\subsection{Functorially finite wide subcategories of $\overline{\mathcal{F}}$}\label{section_example_b}
By Proposition \ref{lemma_ab}, the subcategories $\mathcal{W}_i$ are in bijection with the elements of \ref{itemb} via $\mathcal{W}_i\mapsto \overline{\mathcal{W}}_i:=\text{add}\{ \Sigma^{2\mathbb{Z}}\mathcal{W}_i\}$. That is, the elements of set \ref{itemb} are:
\begin{align*}
    &\overline{\mathcal{W}}_1=\text{add}\{\Sigma^{2\mathbb{Z}} f_1\},\\
    &\overline{\mathcal{W}}_2=\text{add}\{ \Sigma^{2\mathbb{Z}}f_2\}, \\
    &\overline{\mathcal{W}}_3=\text{add}\{\Sigma^{2\mathbb{Z}} f_3\}, \\
    &\overline{\mathcal{W}}_4=\text{add}\{\Sigma^{2\mathbb{Z}} f_4\} \\
    &\overline{\mathcal{W}}_5=\text{add}\{ \Sigma^{2\mathbb{Z}}f_1,\Sigma^{2\mathbb{Z}} f_3\},\\
    &\overline{\mathcal{W}}_6=\text{add}\{\Sigma^{2\mathbb{Z}} f_2, \Sigma^{2\mathbb{Z}}f_4\}, \\
    &\overline{\mathcal{W}}_7=\overline{\mathcal{F}}.
\end{align*}

\subsection{Homological epimorphisms of $2$-homological pairs from $(\Phi,\mathcal{F})$}\label{section_example_c}
The elements of \ref{itema} are also in bijection with the elements of \ref{itemc} via the construction described in Remark \ref{remark_construction6.3}.

For example, for $\mathcal{W}_1=\text{add}\{f_1\}$, we have that the strong $\mathcal{W}_1$-envelopes of the indecomposables in $\mathcal{F}$ are
\begin{align*}
    f_1\xrightarrow{id_{f_1} } f_1,\, f_2\rightarrow 0,\, f_3\rightarrow 0, \, f_4\rightarrow 0.
\end{align*}
Hence, by \cite[Lemma 4.3]{HJV}, we have that the left adjoint of the inclusion functor $\mathcal{W}_1\hookrightarrow \mathcal{F}$ is
\begin{align*}
    \iota_1^*:&\mathcal{F}\rightarrow \mathcal{W}_1\\
    &f_1\mapsto f_1,
    \\ &f_2,f_3,f_4\mapsto 0.
\end{align*}
Then, 
\begin{align*}
   s_1=\iota_1^*(\Phi_\Phi)=f_1 \text{ and }
   \Gamma_1=\End_\Phi (s_1)=\langle id_{f_1} \rangle.
\end{align*}

The homomorphism $\phi_1$ is given by the composition
\begin{align*}
\xymatrix@R=0.5em{
    \Phi\ar[r]^{\sim}& \End_\Phi(\Phi)\ar[r]^{\iota_1^*(-)}& \Gamma_1\\
    1_\Phi\ar@{|->}[r]&(f_1\oplus f_2 \oplus f_3\xrightarrow{id} f_1\oplus f_2 \oplus f_3)\ar@{|->}[r]&(f_1\xrightarrow{id}f_1).
    }
\end{align*}
Recall that $\Gamma_1$ is a right module over $\Phi$. By \cite[Theorem III.1.6]{ASS}, as a representation of the corresponding quiver, we have that $(\Gamma_1)_\Phi$ is
\begin{align*}
    0\longrightarrow 0\longrightarrow k,
\end{align*}
that is $(\Gamma_1)_\Phi\cong f_1$. Then
\begin{align*}
      \mathcal{G}_1= \phi_1^{!}(\mathcal{W}_1)=\Hom_{\Phi}(f_1,\mathcal{W}_1)=\text{add}(\Gamma_1)
\end{align*}
and we have the homological epimorphism of $2$-homological pairs
\begin{align*}
    \phi_1:(\Phi,\mathcal{F})\rightarrow (\Gamma_1, \mathcal{G}_1)
\end{align*}
with $(\phi_1)_*(\mathcal{G}_1)=\mathcal{W}_1$.

Similarly, we can compute the other six homological epimorphisms of $2$-homological pairs starting at $(\Phi,\mathcal{F})$, see Table \ref{table:table_1}.
\begin{table}[H]
\begin{center}
\begin{tabular}{ |c| c| c |c|}
 \hline
 $j$ & $\Gamma_j$ & $\mathcal{G}_j$\ &$\phi_j:(\Phi, \mathcal{F})\rightarrow (\Gamma_j,\mathcal{G}_j)$\\
 \hline & & &\\ 
 1 & $\End_\Phi(f_1)$ & $\text{add}(\Gamma_1)$ & $1_{\Phi}\mapsto id_{f_1}$\\[1em]
 2 & $\End_\Phi(f_2\oplus f_2)$ & $\text{add}(\Gamma_2)$ & $1_{\Phi}\mapsto id_{f_2\oplus f_2}$\\ [1em]
 3 & $\End_\Phi(f_3\oplus f_3)$ & $\text{add}(\Gamma_3)$ & $1_{\Phi}\mapsto id_{f_3\oplus f_3}$\\[1em]
 4 & $\End_\Phi(f_4)$ & $\text{add}(\Gamma_4)$ & $1_{\Phi}\mapsto id_{f_4}$\\[1em]
 5 & $\End_\Phi(f_1\oplus f_3\oplus f_3)$ & $\text{add}(\Gamma_5)$ & $1_{\Phi}\mapsto id_{f_1\oplus f_3\oplus f_3}$\\[1em]
 6 & $\End_\Phi(f_2\oplus f_2\oplus f_4)$ & $\text{add}(\Gamma_6)$ & $1_{\Phi}\mapsto id_{f_2\oplus f_2\oplus f_4}$\\[1em]
 7 & $\End_\Phi (\Phi_\Phi)\cong \Phi$ & $\mathcal{F}$ & $1_\Phi\mapsto id_{\Phi_\Phi}$\\
 \hline
\end{tabular}
\end{center}
\caption{The homological epimorphisms of $2$-homological pairs from $(\Phi,\mathcal{F})$.}
\label{table:table_1}
\end{table}

\begin{remark}
Note that for $j=2,\, 3,\, 5$ and $6$, the morphism $\phi_j$ is an example of a homological epimorphism of algebras that is not surjective.
%Considering $j=2$, we have that
%\begin{align*}
%\xymatrix@R=0.3em{
%     \End_\Phi(\Phi)\ar[r]^-{\iota_2^*(-)}& \Gamma_2=\End_\Phi(f_2\oplus f_2)\\
%    id_{f_1} \ar@{|->}[r]& (f_2\oplus f_2{\xrightarrow{\begin{psmallmatrix}1&0\\0&0\end{psmallmatrix}}}f_2\oplus f_2),\\
%    id_{f_2} \ar@{|->}[r]& (f_2\oplus f_2{\xrightarrow{\begin{psmallmatrix}0&0\\0&1\end{psmallmatrix}}}f_2\oplus f_2),\\
%    (f_1\rightarrow f_2) \ar@{|->}[r]& (f_2\oplus f_2{\xrightarrow{\begin{psmallmatrix}0&0\\1&0\end{psmallmatrix}}}f_2\oplus f_2),\\
%    id_{f_3} \ar@{|->}[r]& 0,\\
%    (f_2\rightarrow f_3) \ar@{|->}[r]& 0.
%}
%\end{align*}
%Hence $f_2\oplus f_2{\xrightarrow{\begin{psmallmatrix}0&1\\0&0\end{psmallmatrix}}}f_2\oplus f_2$ is an element in $\Gamma_2$ which does not belong to the image of $\phi_2$.

\end{remark}

\subsection{Universal localizations of $(\Phi,\mathcal{F})$}\label{section_example_d}
The elements in \ref{itemc} are in bijection with the elements in \ref{itemd}. In fact, each homological epimorphism of $2$-homological pairs $\phi_j$, is also the universal localization of $(\Phi,\mathcal{F})$ with respect to
\begin{align*}
    \mathcal{U}_j:={}^{\perp}((\phi_j)_*(\mathcal{D}^b(\mmod \Gamma_j))).
\end{align*}
So it remains to compute $\mathcal{U}_j$ for each $1\leq j\leq 7$. In order to do so, we need to understand $\mathcal{D}^b(\mmod\Phi)$ first.

Consider the path algebra $kA_3$ of the quiver
\begin{align*}
\xymatrix{
3\ar[r]& 2\ar[r]& 1.
}
\end{align*}
Denoting the indecomposable modules in $\mmod kA_3$ by their radical series, the Auslander–Reiten quiver of $\mathcal{D}^b(\mmod kA_3)$ is

%\begin{align*}
%\xymatrix @!0 {
%&& {\begin{smallmatrix}3\\2\\1\end{smallmatrix}}\ar[rd]\\
%& {\begin{smallmatrix}2\\1\end{smallmatrix}} \ar[ru]\ar[rd] && {\begin{smallmatrix}3\\2\end{smallmatrix}}\ar[rd]\\
%{\begin{smallmatrix}1\end{smallmatrix}}\ar[ru]&& {\begin{smallmatrix}2\end{smallmatrix}}\ar[ru] && {\begin{smallmatrix}3\end{smallmatrix}}.
%}
%\end{align*}

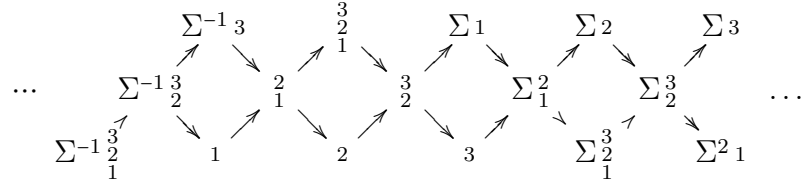
\begin{align*}
\xymatrix @!0{
&&&&\Sigma^{-1}{\begin{smallmatrix}3\end{smallmatrix}}\ar[rd]&& {\begin{smallmatrix}3\\2\\1\end{smallmatrix}}\ar[rd]&& \Sigma {\begin{smallmatrix}1\end{smallmatrix}}\ar[rd] &&\Sigma {\begin{smallmatrix}2\end{smallmatrix}}\ar[rd] && \Sigma {\begin{smallmatrix}3\end{smallmatrix}}\\
&\cdots&&\Sigma^{-1}{\begin{smallmatrix}3\\2\end{smallmatrix}}\ar[ru]\ar[rd]&& {\begin{smallmatrix}2\\1\end{smallmatrix}} \ar[ru]\ar[rd] && {\begin{smallmatrix}3\\2\end{smallmatrix}}\ar[rd]\ar[ru]&&\Sigma{\begin{smallmatrix}2\\1\end{smallmatrix}}\ar[ru]\ar[rd]&&\Sigma {\begin{smallmatrix}3\\2\end{smallmatrix}}\ar[rd]\ar[ru]&&\dots\\
&&\Sigma^{-1}{\begin{smallmatrix}3\\2\\1\end{smallmatrix}}\ar[ru]&&{\begin{smallmatrix}1\end{smallmatrix}}\ar[ru]&& {\begin{smallmatrix}2\end{smallmatrix}}\ar[ru] && {\begin{smallmatrix}3\end{smallmatrix}}\ar[ru]&&\Sigma {\begin{smallmatrix}3\\2\\1\end{smallmatrix}}\ar[ru]&& \Sigma^2{\begin{smallmatrix}1\end{smallmatrix}}
}
\end{align*}

Note that since $kA_3$ is hereditary, every indecomposable in $\mathcal{D}^b(\mmod kA_3)$ is a shift of one of the six indecomposables in $\mmod kA_3$.
Note that $t:={\begin{smallmatrix}1\end{smallmatrix}}\oplus {\begin{smallmatrix}3\\2\\1\end{smallmatrix}}\oplus {\begin{smallmatrix}3\end{smallmatrix}}$ is a tilting module in $\mmod kA_3$  in the sense of \cite[pp 118]{HD}. Let $\Lambda=\End_{kA3}(t)$, it is easy to check that $\Lambda\cong \Phi$.
Then, by \cite[Theorem 1.6 and its Corollary]{HD2}, there is a derived equivalence
\begin{align*}
    R\Hom_{kA_3}(t,-):\mathcal{D}^b(\mmod kA_3)\rightarrow \mathcal{D}^b(\mmod\Phi).
\end{align*}
Under this equivalence, we have 
\begin{align*}
    {\begin{smallmatrix}1\end{smallmatrix}}\mapsto f_1,\, {\begin{smallmatrix}3\\2\\1\end{smallmatrix}}\mapsto f_2, \, {\begin{smallmatrix}3\\2\end{smallmatrix}}\mapsto s_2,\, {\begin{smallmatrix}3\end{smallmatrix}}\mapsto f_3,\, \Sigma{\begin{smallmatrix}2\end{smallmatrix}}\mapsto f_4, \, \Sigma{\begin{smallmatrix}2\\ 1\end{smallmatrix}}\mapsto x,
\end{align*}
where $x$ is an indecomposable in $\mathcal{D}^b(\mmod\Phi)$ that is not isomorphic to a shift of a module in $\mmod\Phi$ and is concentrated in homological degrees 1 and 0. Note that such an indecomposable exists because $\Phi$ is not hereditary.
The Auslander-Reiten quiver of $\mathcal{D}^b(\mmod\Phi)$ is then

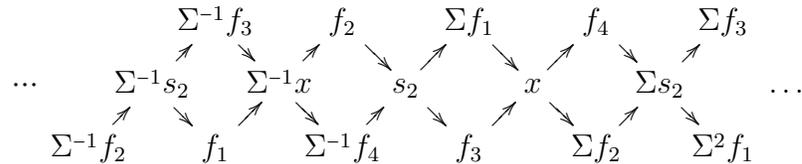
\begin{align*}
\xymatrix @!0{
&&&&\Sigma^{-1}f_3\ar[rd]&& f_2\ar[rd]&& \Sigma f_1\ar[rd] &&f_4\ar[rd] && \Sigma f_3\\
&\cdots&&\Sigma^{-1}s_2\ar[ru]\ar[rd]&& \Sigma^{-1 }x\ar[ru]\ar[rd] && s_2\ar[rd]\ar[ru]&&x\ar[ru]\ar[rd]&&\Sigma s_2\ar[rd]\ar[ru]&&\dots\\
&&\Sigma^{-1}f_2\ar[ru]&&f_1\ar[ru]&& \Sigma^{-1}f_4\ar[ru] && f_3\ar[ru]&&\Sigma f_2\ar[ru]&& \Sigma^2 f_1
}
\end{align*}

We can now compute the wide subcategories $\mathcal{U}_j$.
%In order to compute $\mathcal{U}_i$, we study it and $\mathcal{U}_i^\perp$ corresponding subcategories in $\mathcal{D}^b(\mmod kA_3)$ and then go back to $\mathcal{D}^b(\mmod\Phi)$ through the above equivalence.

We have that 
\begin{align*}
    \mathcal{U}_1^\perp=(\phi_1)_*(\mathcal{D}^b(\mmod\Gamma_1))= \text{add}\{ \Sigma^\mathbb{Z} f_1 \}\subseteq \mathcal{D}^b(\mmod\Phi).
\end{align*}

The indecomposables in $\mathcal{D}^b(\mmod\Phi)$ that have zero morphisms to $f_1$ and all its shifts are
  $ \Sigma^{\mathbb{Z}} f_2, \, \Sigma^{\mathbb{Z}} x \text{ and } \Sigma^{\mathbb{Z}} f_3$. Hence $\phi_1$ is the universal localization of $(\Phi,\mathcal{F})$ with respect to
\begin{align*}
    \mathcal{U}_1=\text{add}(\Sigma^\mathbb{Z}\{ f_2,f_3,x \}).
\end{align*}

Similarly, for $2\leq j\leq 7$, we can compute the wide subcategory $\mathcal{U}_j$ associated to the universal localization $\phi_j$, see Table \ref{table:table_2}.

\begin{table}[H]
\begin{center}
\begin{tabular}{ |c| c| c |}
 \hline
 $j$ & $\mathcal{U}_j^\perp$ & $\mathcal{U}_j$ \\
  \hline & & \\ 
 1& $\text{add}\{ \Sigma^\mathbb{Z} f_1 \}$ & $\text{add}\{\Sigma^\mathbb{Z}\{ f_2,f_3,x \}\}$\\[1em]
 2& $\text{add}\{ \Sigma^\mathbb{Z} f_2 \}$ & $\text{add}\{ \Sigma^\mathbb{Z} \{ f_3,f_4,s_2 \} \}$\\[1em]
 3& $\text{add}\{ \Sigma^\mathbb{Z} f_3 \}$ & $\text{add}\{ \Sigma^\mathbb{Z} \{ f_1,f_4,x \} \}$\\[1em]
 4&$\text{add}\{ \Sigma^\mathbb{Z} f_4 \}$ & $\text{add}\{ \Sigma^\mathbb{Z} \{ f_1,f_2,s_2 \} \}$\\[1em]
 5& $\text{add}\{ \Sigma^\mathbb{Z} \{ f_1,f_3\} \}$ & $\text{add}\{ \Sigma^\mathbb{Z} x \}$\\[1em]
 6& $\text{add}\{ \Sigma^\mathbb{Z} \{ f_2,f_4\} \}$ & $\text{add}\{ \Sigma^\mathbb{Z} s_2 \}$\\[1em]
 7& $\mathcal{D}^b(\mmod\Phi)$ & $0$\\
 \hline
\end{tabular}
\end{center}
\caption{The wide subcategories $\mathcal{U}_j$.}
\label{table:table_2}
\end{table}

\end{document}